\documentclass[11pt]{amsart}

\usepackage{amsfonts,amssymb,amsmath,amsthm,amscd,enumerate}
\usepackage{latexsym}
\usepackage{euscript}

\newtheorem{theorem}{Theorem}[section]

\newtheorem{lemma}[theorem]{Lemma}
\newtheorem{corollary}[theorem]{Corollary}
\newtheorem{proposition}[theorem]{Proposition}

\title[The torsion subgroup of the additive group of a ring]{
The torsion subgroup of the additive group of a Lie nilpotent associative ring of class $3$}

\author{Galina Deryabina}

\address{Department of Computational Mathematics and Mathematical Physics (FS-11), Bauman Moscow State Technical University, 2-nd Baumanskaya Street, 5, 105005 Moscow, Russia}

\email{galina\_deryabina@mail.ru}

\author{Alexei Krasilnikov}

\address{Departamento de Matem\'atica, Universidade de Bras\'\i lia, 70910-900 Bras\'\i lia, DF, Brazil}

\email{alexei@unb.br}

\date{}

\begin{document}

\begin{abstract}
Let $\mathbb Z \langle X \rangle$ be the free unital associative ring freely generated by an infinite countable set $X = \{ x_1,x_2, \dots \}$. Define a left-normed commutator $[x_1,x_2, \dots , x_n]$ by $[a,b] = ab - ba$, $[a,b,c] = [[a,b],c]$. For $n \ge 2$, let $T^{(n)}$ be the two-sided ideal in $\mathbb Z \langle X \rangle$ generated by all commutators $[a_1,a_2, \dots , a_n]$ $( a_i \in \mathbb Z \langle X \rangle )$. Let $T^{(3,2)}$ be the two-sided ideal of the ring $\mathbb Z \langle X \rangle$ generated by all elements $[a_1, a_2, a_3, a_4]$ and $[a_1, a_2] [a_3, a_4, a_5]$ $(a_i \in \mathbb Z \langle X \rangle)$.

It has been recently proved in \cite{Kras12} that the additive group of $\mathbb Z \langle X \rangle / T^{(4)}$ is a direct sum $ A \oplus B$ where $A$ is a free abelian group isomorphic to the additive group of $\mathbb Z \langle X \rangle / T^{(3,2)}$ and $B = T^{(3,2)} /T^{(4)}$ is an elementary abelian $3$-group. A basis of the free abelian summand $A$ was described explicitly in \cite{Kras12}. The aim of the present article is to find a basis of the elementary abelian $3$-group $B$.
\end{abstract}

\maketitle

\section{Introduction}

Let $\mathbb Z \langle X \rangle$ be the free unital associative ring freely generated by an infinite countable set $X = \{ x_i \mid i \in \mathbb N \}$. Then $\mathbb Z \langle X \rangle$ is the free $\mathbb Z$-module with a basis $\{ x_{i_1} x_{i_2} \dots x_{i_k} \mid k \ge 0, i_l \in \mathbb N \}$ formed by the non-commutative monomials in $x_1, x_2, \dots .$ Define a left-normed commutator $[x_1,x_2, \dots , x_n]$ by $[a,b] = ab - ba$, $[a,b,c] = [[a,b],c]$. For $n \ge 2$, let $T^{(n)}$ be the two-sided ideal in $\mathbb Z \langle X \rangle$ generated by all commutators $[a_1,a_2, \dots , a_n]$ $( a_i \in \mathbb Z \langle X \rangle )$. Note that the quotient ring $\mathbb Z \langle X \rangle / T^{(n)}$is the universal Lie nilpotent associative ring of class $n-1$ generated by $X$.

It is clear that the quotient ring $\mathbb Z \langle X \rangle /T^{(2)}$ is isomorphic to the ring $\mathbb Z [X]$ of commutative polynomials in $x_1, x_2, \dots .$ Hence, the additive group of $\mathbb Z \langle X \rangle /T^{(2)}$ is free abelian  and its basis is formed by the (commutative) monomials. Recently Bhupatiraju, Etingof, Jordan, Kuszmaul and Li \cite{BEJKL12} have proved that the additive group of $\mathbb Z \langle X \rangle /T^{(3)}$ is also free abelian and found explicitly its basis \cite[Prop. 3.2]{BEJKL12}. Very recently in \cite{Kras12} the second author of the present article has proved that the additive group of $\mathbb Z \langle X \rangle / T^{(4)}$ is a direct sum $A \oplus B$ of a free abelian group $A$ and an elementary abelian $3$-group $B$.

More precisely, let $T^{(3,2)}$ be the two-sided ideal of the ring $\mathbb Z \langle X \rangle$ generated by all elements $[a_1, a_2, a_3, a_4]$ and $[a_1, a_2] [ a_3, a_4, a_5]$ where $a_i \in \mathbb Z \langle X \rangle$. Clearly, $T^{(4)} \subset T^{(3,2)}$. It has been proved in \cite{Kras12} that $T^{(3,2)} / T^{(4)}$ is a non-trivial elementary abelian $3$-group and the additive group of $\mathbb Z \langle X \rangle / T^{(3,2)}$ is free abelian. It follows that the additive group of $\mathbb Z \langle X \rangle / T^{(4)}$ is a direct sum $A \oplus B$ where $B = T^{(3,2)}/T^{(4)}$ is an elementary abelian $3$-group and $A$ is a free abelian group isomorphic to the additive group of $\mathbb Z \langle X \rangle / T^{(3,2)}$. A $\mathbb Z$-basis of the additive group of $\mathbb Z \langle X \rangle / T^{(3,2)}$ was explicitly written in \cite[Lemma 5.6]{Kras12}; this basis can be easily deduced from the results of either \cite{EKM09} or \cite{Gordienko07} or \cite{Volichenko78}. The aim of the present article is to find a basis over $\mathbb F_3 = \mathbb Z / 3 \, \mathbb Z$ of the elementary abelian $3$-group $T^{(3,2)} / T^{(4)}$.

Let $v_k = [x_1, x_2] \dots [ x_{2k-1}, x_{2k}] [x_{2k+1}, x_{2k+2}, x_{2k+3}]$ $(k \ge 1)$. One can deduce from the results of \cite{EKM09, Gordienko07, Volichenko78} that $3 \, v_1 \in T^{(4)}$; on the other hand, it was proved in \cite[Theorem 1.1]{Kras12} that $v_1 \notin T^{(4)}$. Since $3 \, v_1 \in T^{(4)}$, we have $3 \, v_k \in T^{(4)}$ for all $k$. Our first main result is as follows.

\begin{theorem}\label{maintheorem1}
For all $k \ge 1$, $v_k \notin T^{(4)}$.
\end{theorem}

Our second main result describes the $3$-torsion subgroup $T^{(3,2)}/T^{(4)}$ of the additive group of $\mathbb Z \langle X \rangle / T^{(4)}$. Let
\begin{multline*}
\mathcal E = \Big\{ x_{j_1} \dots x_{j_l} \, [x_{i_1}, x_{i_2}] \dots [ x_{i_{2k-1}}, x_{i_{2k}}] [x_{i_{2k+1}}, x_{i_{2k+2}}, x_{i_{2k+3}}] \mid  \\
l \ge 0, \, k \ge 1, \, j_1 \le j_2 \le \dots \le j_l; \, i_1 < i_2 < \dots < i_{2k+3} \Big\} .
\end{multline*}

\begin{theorem}\label{maintheorem2}
The set $\{ e + T^{(4)} \mid e \in \mathcal E \}$ is a basis of the elementary abelian $3$-group $T^{(3,2)}/T^{(4)}$ over $\mathbb F_3 = \mathbb Z / 3 \, \mathbb Z$.
\end{theorem}

This theorem confirms the conjecture concerning a basis of $T^{(3,2)}/T^{(4)}$ over $\mathbb F_3$ made in \cite[Remark 1.7]{Kras12}.

To prove Theorem \ref{maintheorem1} we use the following result that may be of independent interest.

\begin{theorem} \label{maintheorem3}
Let $K$ be an arbitrary unital associative and commutative ring and let $K \langle Y \rangle$ be the free associative $K$-algebra on a non-empty set $Y$ of free generators. Let $T^{(4)}$ be the two-sided ideal in $K \langle Y \rangle$ generated by all commutators $[a_1,a_2, a_3, a_4]$ $( a_i \in K \langle Y \rangle )$. Then the ideal $T^{(4)}$ is generated by the polynomials
\begin{gather}
\label{c4} [y_{1}, y_{2}, y_{3}, y_{4}] \qquad (y_i \in Y ),
\\
\label{c33} [y_{1}, y_{2}, y_{3}] [y_{4}, y_{5}, y_{6} ] \qquad (y_i \in Y ),
\\
\label{c32-1} [y_{1}, y_{2}] [y_{3}, y_{4}, y_{5}] + [y_{{1}}, y_{{5}}] [y_{{3}}, y_{{4}}, y_{{2}}] \qquad (y_i \in Y ),
\\
\label{c32-2} [y_{1}, y_{2}] [y_{3}, y_{4}, y_{5}] + [y_{{1}}, y_{{4}}] [y_{{3}}, y_{{2}}, y_{{5}}] \qquad (y_i \in Y ),
\\
\label{c222} ([y_{1}, y_{2}] [y_{3}, y_{4}] + [y_{{1}}, y_{{3}}] [y_{{2}}, y_{{4}}]) [y_{{5}}, y_{{6}}] \qquad (y_i \in Y).
\end{gather}
\end{theorem}

A generating set of the ideal $T^{(4)}$ can be rewritten in the following more convenient way:

\begin{corollary}\label{corollarytomaintheorem3}
The ideal $T^{(4)}$ is generated (as a two-sided ideal in $K \langle Y \rangle$) by the polynomials (\ref{c4}), (\ref{c33}) together with the polynomials
\begin{gather}
[y_{1}, y_{2}] [y_{3}, y_{4}, y_{5}] - \mbox{\rm sgn} \, (\sigma ) [y_{{\sigma (1)}}, y_{{\sigma (2)}}] [y_{{\sigma (3)}}, y_{{\sigma (4)}}, y_{{\sigma (5)}}] \label{c32sigma}
\\
(y_i \in Y  , \ \sigma \in S_5), \nonumber
\\
[y_{1}, y_{2}] [y_{3}, y_{4}] [y_{5}, y_{6}] - \mbox{\rm sgn} \, (\sigma ) [y_{{\sigma (1)}}, y_{{\sigma (2)}}] [y_{{\sigma (3)}}, y_{{\sigma (4)}}] [y_{{\sigma (5)}}, y_{{\sigma (6)}}] \label{c222sigma}
\\
(y_i \in Y , \ \sigma \in S_6 ). \nonumber
\end{gather}
\end{corollary}

Note that $3 \, [y_1,y_2][y_3,y_4,y_5] \in T^{(4)}$ for all $y_i \in Y$; one can deduce this from the results of \cite{EKM09, Gordienko07, Volichenko78}, see also \cite[Corollary 2.4]{Kras12}. Hence, if $\frac{1}{3} \in K$ then all polynomials
\begin{equation}
\label{23}
[y_1,y_2][y_3,y_4,y_5] \qquad (y_i \in Y )
\end{equation}
belong to $T^{(4)}$. Since the polynomials (\ref{c33})--(\ref{c32-2}) belong to the ideal generated by the polynomials (\ref{23}), Theorem \ref{maintheorem3} implies the following corollary that has been proved in \cite{EKM09, Volichenko78}.

\begin{corollary}[see \cite{EKM09, Volichenko78}]
If $\frac{1}{3} \in K$ then $T^{(4)}$ is generated as a two-sided ideal of $K \langle Y \rangle$ by the polynomials (\ref{c4}), (\ref{c222}) and (\ref{23}).
\end{corollary}

\textbf{Remarks.} The study of Lie nilpotent associative rings and algebras was started by Jennings \cite{Jennings47} in 1947. Jennings' work was motivated by the study of modular group algebras of finite $p$-groups. Since then Lie nilpotent associative rings and algebras were investigated in various papers from various points of view; see, for instance, \cite{AS01, Gordienko07, GTS10, GL83, Kras97, Latyshev65, Petrogradsky11, RT99, Volichenko78} and the bibliography there.

Recent interest in Lie nilpotent associative algebras was motivated by the study of the quotients $L_{i}/L_{i+1}$ of the lower central series
\[
L_1 > L_2 > \dots > L_n > \dots
\]
of the associated Lie algebra of a free associative algebra $A$; here $L_n$ is the linear span in $A$ of the set of all commutators $[a_1, a_2, \dots , a_n]$ $(a_i \in A)$. The study of these quotients $L_i /L_{i+1}$ was initiated in 2007 in a pioneering article of Feigin and Shoikhet \cite{FS07}; further results on this subject can be found, for example, in \cite{AJ10,BB11,BJ10,DE08,DKM08,EKM09,JO13,Kerchev13}. Since $T^{(n)}$ is the ideal in $A$ generated by $L_n$, some results about the quotients $T^{(i)}/T^{(i+1)}$ were obtained in these articles as well; in \cite{EKM09,Kerchev13} the latter quotients were the primary objects of study.

In 2012 Bhupatiraju, Etingof, Jordan, Kuszmaul and Li \cite{BEJKL12} started the study of the quotients $L_i /L_{i+1}$ for free associative rings $A$; that is, for free associative algebras $A$ over $\mathbb Z$. In this case the quotients in question (may) develop torsion and one of the objects of study is the pattern of this torsion.  In \cite{BEJKL12} many results concerning torsion in $L_i/L_{i+1}$ were obtained and various open problems concerning this torsion were posed. Various results about the quotients $T^{(i)}/T^{(i+1)}$ and $A / T^{(i+1)}$ were obtained also in \cite{BEJKL12}; in particular, as mentioned above, it was proved that the additive group of $\mathbb Z \langle X \rangle /T^{(3)}$ is free abelian and a basis of this group was exhibited \cite[Prop. 3.2]{BEJKL12}. Some further results concerning the quotients $T^{(i)} / T^{(i+1)}$ for various associative rings $A$ were obtained by Cordwell, Fei and Zhou in \cite{CFZ13}.

It was observed in \cite{BEJKL12} that in the quotients $T^{(i)} / T^{(i+1)}$ there is no torsion in the (additive) subgroups generated by the polynomials of small degree. However, in \cite{Kras12} the second author of the present article has proved that the image of $v_1 = [x_1,x_2] [x_3,x_4,x_5]$ in $\mathbb Z \langle X \rangle /T^{(4)}$ is an element of order $3$, that the torsion subgroup of $\mathbb Z \langle X \rangle /T^{(4)}$ coincides with $T^{(3,2)} / T^{(4)}$ and that the latter is an elementary abelian $3$-group. To find a basis of this group is the aim of the present article.

Interesting computational data about the torsion subgroup of $T^{(i)} / T^{(i+1)}$ for various $i$ was presented in \cite{CFZ13}. In particular, this data suggests that the additive group of $\mathbb Z \langle X \rangle /T^{(5)}$ may have no torsion. Whether this group is torsion-free indeed is still an open problem. In \cite{CK13} Costa and the second author of the present article have found generators for the ideal $T^{(5)}$ of the free associative algebra $K \langle X \rangle$. This result is similar to Theorem \ref{maintheorem3}, which gives such generators for the ideal $T^{(4)}$. The result of \cite{CK13} might be the first step in proving that the additive  group of $\mathbb Z \langle X \rangle /T^{(5)}$ is torsion-free.

\section{Preliminaries}

Let $K$ be an arbitrary associative and commutative unital ring. An ideal $T$ of the free $K$-algebra $K \langle X \rangle$ is called \textit{$T$-ideal} if $\phi (T) \subseteq T$ for all endomorphisms $\phi$ of $K \langle X \rangle$. One can easily see that $T^{(4)}$ and $T^{(3,2)}$ are $T$-ideals. We refer to \cite{Drenskybook, GZbook, Rowenbook} for terminology, basic facts and references concerning $T$-ideals and polynomial identities in associative algebras.

Let $R^{(i_1, i_2, \dots )}$ denote the $K$-linear span in $K \langle X \rangle$ of all monomials of multi-degree $(i_1, i_2, \dots )$, that is, of the monomials of degree $i_1$ in $x_1$, $i_2$ in $x_2$ etc. Let $f \in K \langle X \rangle$ be a polynomial, $f = f^{(i_{11}, i_{12} \dots )} + \dots + f^{(i_{s1}, i_{s2} \dots )}$ where $f^{(i_{t1}, i_{t2}, \dots )} \in R^{(i_{t1}, i_{t2}, \dots )}$ for all $t$. We say that the polynomials $f^{(i_{11}, i_{12} \dots )},$  $\dots ,$ $f^{(i_{s1}, i_{s2} \dots )}$ are \textit{multihomogeneous components} of the polynomial $f$. For instance, if $f = x_1^3 + 2 x_1x_2x_1 - x_1^2 x_2$ then $x_1^3$ and $2 x_1x_2x_1 - x_1^2 x_2$ are the multihomogeneous components of $f$ of multidegrees $(3, 0, \dots)$ and $(2,1,0, \dots)$, respectively.

We say that an ideal $I$ of $K \langle X \rangle$ is \textit{multihomogeneous} if, for each $f \in I$, all multihomogeneous components of $f$ also belong to $I$. The ideal $T^{(4)}$ is multihomogeneous because it is spanned by multihomogeneous polynomials $a_0[a_1,a_2,a_3,a_4]a_5$ where all $a_i$ are monomials. Similarly, the ideal $T^{(3,2)}$ is also multihomogeneous.

Let $\Gamma$ be the unital subalgebra of the free $K$-algebra $K \langle X \rangle$ generated by all (left-normed) commutators $[x_{i_1}, x_{i_2}, \dots , x_{i_l}]$ where $l \ge 2$ and $x_i \in X$ for all $i$. The following assertion is well-known for an arbitrary $T$-ideal $T$ in $K \langle X \rangle$ if $K$ is an infinite field (see, for instance, \cite[Proposition 4.3.3]{Drenskybook}). However, its proof remains valid over an arbitrary associative and commutative unital ring $K$ if the $T$-ideal $T$ is \textit{ multihomogenous}.

\begin{proposition}[see \cite{Drenskybook}]
Let $K$ be an associative and commutative unital ring and let $T$ be a multihomogeneous $T$-ideal in the free unital associative algebra $K \langle X \rangle$. Then $T$ is generated by the set $T \cap \Gamma$ as a \textbf{two-sided} ideal of $K \langle X \rangle$.
\end{proposition}

Note that if $I$ is a two-sided ideal of $K \langle X \rangle$ then, for all $i$,
\[
[(I \cap \Gamma) ,x_i] \subseteq (I \cap \Gamma).
\]
It follows that if the set $I \cap \Gamma$ generates $I$ as a \textit{two-sided} ideal of $K \langle X \rangle$ then $I \cap \Gamma$ generates $I$ as a \textit{left} ideal as well. Thus, we have
\begin{corollary}\label{TcapGammaleft}
Let $K$ be an associative and commutative unital ring and let $T$ be a multihomogeneous $T$-ideal in the free unital associative algebra $K \langle X \rangle$. Then $T$ is generated by the set $T \cap \Gamma$ as a \textbf{left} ideal of $K \langle X \rangle$.
\end{corollary}

The following lemma is well-known.
\begin{lemma}
\label{directsum}
The additive group of the ring $\mathbb Z \langle X \rangle$ is a direct sum of the subgroups $x_{i_1} x_{i_2} \dots x_{i_s} \Gamma$ where $s \ge 0,$ $ i_1 \le i_2 \le \dots \le i_s$,
\[
\mathbb Z \langle X \rangle = \bigoplus_{s \ge 0; \ i_1 \le \dots \le i_s} x_{i_1} x_{i_2} \dots x_{i_s} \Gamma .
\]
\end{lemma}
This equality can be rewritten as follows:
\begin{equation}\label{directsum2}
\mathbb Z \langle X \rangle= \Gamma \oplus \bigoplus_{s \ge 1; \ i_1 \le \dots \le i_s} x_{i_1} x_{i_2} \dots x_{i_s} \Gamma .
\end{equation}
\begin{lemma}\label{IcapGamma}
Let $I$ be a left ideal of $\mathbb Z \langle X \rangle$ generated (as a left ideal in $\mathbb Z \langle X \rangle$) by a set $S \subset \Gamma$. Then
\[
I = \Gamma \cdot S \oplus \bigoplus_{s \ge 1; \ i_1 \le \dots \le i_s} x_{i_1} x_{i_2} \dots x_{i_s} \Gamma \cdot S
\]
where $\Gamma \cdot S$ is the left ideal of the ring $\Gamma$ generated by $S$. In particular, $I \cap \Gamma = \Gamma \cdot S$, that is, $I \cap \Gamma$ is the left ideal of $\Gamma$ generated by $S$.
\end{lemma}
\begin{proof}
We have
\[
\mathbb Z \langle X \rangle \cdot S = \Gamma \cdot S + \sum_{s \ge 1; \ i_1 \le \dots \le i_s} x_{i_1} x_{i_2} \dots x_{i_s} \Gamma \cdot S.
\]
Since $S \subset \Gamma$, we have
\[
x_{i_1} x_{i_2} \dots x_{i_s} \Gamma \cdot S \subset x_{i_1} x_{i_2} \dots x_{i_s} \Gamma
\]
for all $s \ge 0; \ i_1 \le \dots \le i_s$. The result follows from (\ref{directsum2}).
\end{proof}
Lemma \ref{IcapGamma} immediately implies the following:
\begin{corollary}\label{GammaIcapGamma}
Under the hypothesis of Lemma \ref{IcapGamma}, the (additive) group $\mathbb Z \langle X \rangle / I$ is isomorphic to the direct sum of the groups
\[
 (x_{i_1} x_{i_2} \dots x_{i_s} \Gamma) / (x_{i_1} x_{i_2} \dots x_{i_s} \Gamma \cdot S) \qquad (s \ge 0; \ i_1 \le \dots \le i_s),
\]
that is,
\[
\mathbb Z \langle X \rangle / I \simeq \Gamma / ( \Gamma \cdot S)  \oplus \bigoplus_{s \ge 1; \ i_1 \le \dots \le i_s} ( x_{i_1} x_{i_2} \dots x_{i_s} \Gamma) / ( x_{i_1} x_{i_2} \dots x_{i_s} \Gamma \cdot S).
\]
For all $s \ge 1$, $i_1 \le \dots \le i_s$, the group $( x_{i_1} x_{i_2} \dots x_{i_s} \Gamma) / ( x_{i_1} x_{i_2} \dots x_{i_s} \Gamma \cdot S)$ is isomorphic to $\Gamma / ( \Gamma \cdot S)$ with an isomorphism induced by the mapping $\Gamma \rightarrow x_{i_1} x_{i_2} \dots x_{i_s} \Gamma$ that maps $f \in \Gamma$ to $ x_{i_1} x_{i_2} \dots x_{i_s} f$.
\end{corollary}

Let $P_n$ $(n \ge 1)$ be the subgroup of the additive group of $\mathbb Z \langle X \rangle$ generated by all monomials which are of degree $1$ in each variable $x_1, \dots , x_n$ and do not contain any other variable. Then $P_n$ is a free abelian group of rank $n!$.

The following well-known theorem (see, for example, \cite[Theorem 4.3.9]{Drenskybook}) describes a certain basis $S$ (called Specht basis) of the free abelian group $P_n \cap \Gamma$. Note that in \cite{Drenskybook} this theorem is stated for the vector subspace $P_n \cap \Gamma$ of the free algebra $K \langle X \rangle$ over a field $K$. However, it remains valid for the additive group $P_n \cap \Gamma$ of the free ring $\mathbb Z \langle X \rangle$. Indeed, on one hand, it is easy to see that the elements of the set $S$ generate the additive group $P_n \cap \Gamma$. On the other hand, the set $S$ is linearly independent over $\mathbb Q$ in $\mathbb Q \langle X \rangle$, therefore, it is linearly independent over $\mathbb Z$ in $\mathbb Z \langle X \rangle \subset \mathbb Q \langle X \rangle$.

\begin{theorem}[see \cite{Drenskybook}]\label{Spechtbasis}
A basis $S$ of the free abelian group $P_n \cap \Gamma$ $(n \ge 2)$ consists of all products $c_1 c_2 \dots c_m$ $(m \ge 1)$ of commutators $c_i$ $(1 \le i \le m)$ such that
\begin{itemize}
\item[(i)] Each product $c_1 c_2 \dots c_m$ is multilinear in the variables $x_1, \dots , x_n$;
\item[(ii)] Each factor $c_i = [x_{p_1}, x_{p_2}, \dots , x_{p_s}]$ is a left-normed commutator of length $\ge 2$ and the maximal index is in the first position, that is, $p_1 > p_j$ for all $j$, $2 \le j \le s$;
\item[(iii)] In each product the shorter commutators precede the longer, that is, the length of $c_k$ is smaller than or equal to the length of $c_{k+1}$ $(1 \le k \le m-1)$;
%
\item[(iv)] If two consecutive factors are commutators of equal length then the first variable of the first commutator is smaller than the first variable in the second one, that is, if
\[
c_k = [x_{p_1}, x_{p_2}, \dots , x_{p_s}], \qquad c_{k+1} = [x_{q_1}, x_{q_2}, \dots , x_{q_s}]
\]
then $p_1 < q_1$.
\end{itemize}
\end{theorem}
For example, the Specht basis $S$ of $P_4 \cap \Gamma$ is as follows: $S = S_1 \cup S_2$ where
\begin{align*}
S_1 = & \bigl\{ [x_4, x_{i_1}, x_{i_2}, x_{i_3}] \mid \{ i_1, i_2, i_3 \} = \{ 1,2,3 \} \bigr\} ,
\\
S_2 = & \bigl\{ [x_{i_1}, x_{i_2}] [x_4, x_{i_3}] \mid \{ i_1, i_2, i_3 \} = \{ 1,2,3 \} , \, i_1 > i_2 \bigr\} .
\end{align*}
The basis $S$ of $P_5 \cap \Gamma$ is as follows: $S = S_1 \cup S_2 \cup S_3$ where
\begin{align*}
S_1 = & \bigl\{ [x_5, x_{i_1}, x_{i_2}, x_{i_3}, x_{i_4}] \mid \{ i_1, i_2, i_3, i_4 \} = \{ 1,2,3, 4 \} \bigr\} ,
\\
S_2 = & \bigl\{ [x_5, x_{i_1}] [x_{i_2}, x_{i_3}, x_{i_4}] \mid \{ i_1, i_2, i_3, i_4 \} = \{ 1,2,3, 4 \} , i_2 > i_3, i_4 \bigr\} ,
\\
S_3 = & \bigl\{ [x_{i_1},x_{i_2}] [x_5, x_{i_3}, x_{i_4}] \mid \{ i_1, i_2, i_3, i_4 \} = \{ 1,2,3, 4 \} , i_1 > i_2 \bigr\} .
\end{align*}

\section{ Proof of Theorem 1.1 }

In this section we will prove Theorem \ref{maintheorem1} assuming that Corollary \ref{corollarytomaintheorem3} holds. If $K = \mathbb Z$ and $Y = X = \{ x_i \mid i \in \mathbb N \}$ then Corollary \ref{corollarytomaintheorem3} can be rewritten as follows:

\begin{corollary}\label{generators_T4_ideal}
The ideal $T^{(4)}$ is generated as a two-sided ideal in $\mathbb Z \langle X \rangle$ by the polynomials
\begin{gather}
\label{cc4} [x_{i_1}, x_{i_2}, x_{i_3}, x_{i_4}] \qquad (i_l \in \mathbb N ),
\\
\label{cc33} [x_{i_1}, x_{i_2}, x_{i_3}] [x_{i_4}, x_{i_5}, x_{i_6}] \qquad (i_l \in \mathbb N ),
\\
\label{cc32} [x_{i_1}, x_{i_2}] [x_{i_3}, x_{i_4}, x_{i_5}] - \mbox{\rm sgn} \, (\sigma ) [x_{i_{\sigma (1)}}, x_{i_{\sigma (2)}}] [x_{i_{\sigma (3)}}, x_{i_{\sigma (4)}}, x_{i_{\sigma (5)}}]
\\
(i_l \in \mathbb N  , \ \sigma \in S_5), \nonumber
\\
\label{cc222} [x_{i_1}, x_{i_2}] [x_{i_3}, x_{i_4}] [x_{i_{5}}, x_{i_{6}}] - \mbox{\rm sgn} \, (\sigma ) [x_{i_{\sigma (1)}}, x_{i_{\sigma (2)}}] [x_{i_{\sigma (3)}}, x_{i_{\sigma (4)}}] [x_{i_{\sigma (5)}}, x_{i_{\sigma (6)}}]
\\
(i_l \in \mathbb N , \ \sigma \in S_6 ). \nonumber
\end{gather}
\end{corollary}

We need the following:

\begin{lemma}\label{generators_T4_left_ideal}
The ideal $T^{(4)}$ is generated as a \textbf{left} ideal in $\mathbb Z \langle X \rangle$ by the polynomials
\begin{equation} \label{cck}
[x_{i_1}, x_{i_2}, \dots , x_{i_k}] \qquad (k \ge 4, \ i_l \in \mathbb N )
\end{equation}
together with the polynomials (\ref{cc33})--(\ref{cc222}).
\end{lemma}

\begin{proof}
Let $I_1$ be the two-sided ideal of $\mathbb Z \langle X \rangle$ generated by all polynomials (\ref{cc4}). It is clear that as a left ideal $I_1$ is generated by the polynomials (\ref{cck}).

Let $I_2$ be the two-sided ideal of $\mathbb Z \langle X \rangle$ generated by all polynomials (\ref{cc4}) and (\ref{cc33}). One can easily check that, modulo $I_1$, each polynomial of the form (\ref{cc33}) is central in $\mathbb Z \langle X \rangle$. It follows that $I_2/I_1$ is generated as a left ideal in $\mathbb Z \langle X \rangle /I_1$ by the images of the polynomials (\ref{cc33}). Hence, $I_2$ is generated as a left ideal in $\mathbb Z \langle X \rangle$ by the polynomials (\ref{cck}) and (\ref{cc33}).

Let $I_3$ be the two-sided ideal of $\mathbb Z \langle X \rangle$ generated by all polynomials (\ref{cc4})--(\ref{cc32}). Since the polynomials (\ref{cc32}) are central modulo $I_2$, $I_3/I_2$ is generated as a left ideal of $\mathbb Z \langle X \rangle /I_2$ by all polynomials (\ref{cc32}). It follows that $I_3$ as a left ideal of $\mathbb Z \langle X \rangle$ is generated by the polynomials (\ref{cck}), (\ref{cc33}), (\ref{cc32}).

Now to complete the proof of Lemma \ref{generators_T4_left_ideal}, that is, to prove that $T^{(4)}$ is generated as a left ideal of $\mathbb Z \langle X \rangle$ by the polynomials (\ref{cck}) together with the polynomials (\ref{cc33})--(\ref{cc222}), it suffices to check that the polynomials (\ref{cc222}) are central modulo $I_3$. Since each polynomial of the form (\ref{cc222}) is, modulo $I_1$, a linear combination of polynomials of the form
\begin{equation}\label{cc222-1}
\bigl( [x_{i_1}, x_{i_2}][x_{i_3}, x_{i_4}] + [x_{i_1}, x_{i_3}][x_{i_2}, x_{i_4}] \bigr) [x_{i_5}, x_{i_6}] \qquad (i_l \in \mathbb N ),
\end{equation}
it suffices to check that each polynomial of the form (\ref{cc222-1}) is central modulo $I_3$. We have
\begin{align}
& \label{2221} \bigl[ \bigl( [x_{i_1}, x_{i_2}][x_{i_3}, x_{i_4}] + [x_{i_1}, x_{i_3}][x_{i_2}, x_{i_4}] \bigr) [x_{i_5}, x_{i_6}], x_{i_7} \bigr]
\\
= & [x_{i_1}, x_{i_2}, x_{i_7}][x_{i_3}, x_{i_4}] [x_{i_5}, x_{i_6}]
+ [x_{i_1}, x_{i_2}][x_{i_3}, x_{i_4},x_{i_7}] [x_{i_5}, x_{i_6}] \nonumber
\\
+ & [x_{i_1}, x_{i_2}][x_{i_3}, x_{i_4}] [x_{i_5}, x_{i_6}, x_{i_7}]
+ [x_{i_1}, x_{i_3}, x_{i_7}][x_{i_2}, x_{i_4}] [x_{i_5}, x_{i_6}] \nonumber
\\
+ & [x_{i_1}, x_{i_3}][x_{i_2}, x_{i_4}, x_{i_7}] [x_{i_5}, x_{i_6}]
+ [x_{i_1}, x_{i_3}][x_{i_2}, x_{i_4}] [x_{i_5}, x_{i_6}, x_{i_7}]  \nonumber
\\
= & f_1 + f_2 + f_3 \nonumber
\end{align}
where
\begin{align*}
f_1 = & [x_{i_1}, x_{i_2}, x_{i_7}][x_{i_3}, x_{i_4}] [x_{i_5}, x_{i_6}] +  [x_{i_1}, x_{i_3}, x_{i_7}][x_{i_2}, x_{i_4}] [x_{i_5}, x_{i_6}],
\\
f_2 = & [x_{i_1}, x_{i_2}][x_{i_3}, x_{i_4},x_{i_7}] [x_{i_5}, x_{i_6}] + [x_{i_1}, x_{i_3}][x_{i_2}, x_{i_4}, x_{i_7}] [x_{i_5}, x_{i_6}],
\\
f_3 = & [x_{i_1}, x_{i_2}][x_{i_3}, x_{i_4}] [x_{i_5}, x_{i_6}, x_{i_7}] + [x_{i_1}, x_{i_3}][x_{i_2}, x_{i_4}] [x_{i_5}, x_{i_6}, x_{i_7}].
\end{align*}

One can easily check that $f_1, f_2 \in I_3$.  Further,
\begin{align*}
f_3 = & \bigl( [x_{i_1}, x_{i_2}][x_{i_3}, x_{i_4}] [x_{i_5}, x_{i_6}, x_{i_7}] + [x_{i_1}, x_{i_7}][x_{i_3}, x_{i_4}] [x_{i_5}, x_{i_6}, x_{i_2}] \bigr)
\\
- & \bigl( [x_{i_1}, x_{i_7}][x_{i_3}, x_{i_4}] [x_{i_5}, x_{i_6}, x_{i_2}] + [x_{i_1}, x_{i_7}][x_{i_2}, x_{i_4}] [x_{i_5}, x_{i_6}, x_{i_3}] \bigr)
\\
+ &  \bigl( [x_{i_1}, x_{i_7}][x_{i_2}, x_{i_4}] [x_{i_5}, x_{i_6}, x_{i_3}] +  [x_{i_1}, x_{i_3}][x_{i_2}, x_{i_4}] [x_{i_5}, x_{i_6}, x_{i_7}] \bigr)
\end{align*}
where each sum in parenthesis belong to $I_3$. Hence, $f_3$ belongs to $I_3$ and so does the polynomial (\ref{2221}). It follows that each polynomial (\ref{cc222-1}) is central modulo $I_3$ so $T^{(4)}$ is generated as a left ideal in $\mathbb Z \langle X \rangle$ by the polynomials (\ref{cck}) and (\ref{cc33})--(\ref{cc222}). The proof of Lemma \ref{generators_T4_left_ideal} is completed.
\end{proof}

Let $\Lambda$ be the set of all (left-normed) commutators $[x_{i_1}, \dots , x_{i_k}]$ where $k \ge 2$ and $i_l \in \mathbb N$ for all $l$. Let $\ell ([x_{i_1}, \dots , x_{i_k}])$ denote the length of the commutator $[x_{i_1}, \dots , x_{i_k}]$, $\ell ([x_{i_1}, \dots , x_{i_k}]) = k$. Recall that $\Gamma$ is the unital subring of $\mathbb Z \langle X \rangle$ generated by $\Lambda$.

Let $\Gamma^{(4)} = T^{(4)} \cap \Gamma$; note that $\Gamma^{(4)}$ is an ideal of the ring $\Gamma$. By Lemmas \ref{IcapGamma} and \ref{generators_T4_left_ideal}, we have
\begin{corollary}\label{generators_T4capGamma}
The ideal $\Gamma^{(4)}$ is generated as a \textbf{left} ideal of $\Gamma$ by the polynomials (\ref{cc33}), (\ref{cc32}), (\ref{cc222}) and (\ref{cck}).
\end{corollary}

Let $\mathcal P$ be the set of all products $c_1 \dots c_m$ where $ m \ge 0, \ c_1, \dots , c_m \in \Lambda, \ \ell (c_1) \le \dots \le \ell (c_m).$ We assume that if $m = 0$ then the product above is equal to $1$ so $1 \in \mathcal P$. One can easily check that the set $\mathcal P$ generates the additive group of $\Gamma$.

The following proposition can be deduced easily from Corollary \ref{generators_T4capGamma}.

\begin{proposition}
\label{generatorsaddgrp}
The additive group of the ideal $\Gamma^{(4)} $ of the ring $\Gamma$ is generated by the following polynomials:
\begin{gather}
\label{ccc4} c_1 \dots c_m \in \mathcal P \quad \mbox{where} \quad m \ge 1, \  \ell (c_m) \ge 4,
\\
\label{ccc33} c_1 \dots c_m \in \mathcal P \ \ \mbox{where} \ \  m \ge 2, \ \ell (c_{m-1}) = \ell (c_m) = 3,
\\
\label{ccc32} f \cdot ([x_{i_1}, x_{i_2}] [ x_{i_3}, x_{i_4}, x_{i_5}] - \mbox{\rm sgn} \, (\tau )
[x_{i_{\tau (1)}}, x_{i_{\tau (2)}}] [ x_{i_{\tau (3)}}, x_{i_{\tau (4)}}, x_{i_{\tau (5)}}])
\\
\mbox{where} \quad f \in \mathcal P, \ i_l \in \mathbb N , \ \tau \in S_5 , \nonumber
\\
\label{ccc222} f \cdot ([x_{i_1}, x_{i_2}] [x_{i_3}, x_{i_4}] [x_{i_{5}}, x_{i_{6}}] - \mbox{\rm sgn} \, (\tau ) [x_{i_{\tau (1)}}, x_{i_{\tau (2)}}] [x_{i_{\tau (3)}}, x_{i_{\tau (4)}}] [x_{i_{\tau (5)}}, x_{i_{\tau (6)}}])
\\
\mbox{where} \quad f \in \mathcal P, \ i_l \in \mathbb N , \ \tau \in S_6 . \nonumber
\end{gather}
\end{proposition}

Note that the elements (\ref{ccc4})--(\ref{ccc222}) are multihomogeneous so the additive group $P_{n} \cap \Gamma^{(4)}$ is generated by all elements  (\ref{ccc4})--(\ref{ccc222}) that belong to $P_{n}$.

Let $W = P_{2k + 3} \cap \Gamma$. Then $W$ is generated (as an additive group) by the products $c_1 \dots c_m$ $(m \ge 1)$ of commutators $c_i \in \Lambda$ $(1 \le i \le m)$ such that $c_1 \dots c_m \in P_{2k+3}$. Note that at least one of the commutators $c_i$ is of odd length.

Let $W'$ be the subgroup of $W$ generated by those products $c_1 \dots c_m \in P_{2k+3}$ of commutators that contain either a commutator $c_i$ of length at least $4$ or at least $2$ commutators $c_i, c_j$ $(i \ne j)$ of length $3$.
One can easily check that $W'$ can be generated by the products $c_1 \dots c_m \in \mathcal P$ that belong to $P_{2k+3}$ and have $\ell (c_m) \ge 4 $ or $\ell (c_{m-1}) = \ell (c_m) = 3$. Note that the latter set of generators of $W'$ is the set of all elements of the forms (\ref{ccc4})--(\ref{ccc33}) that belong to $P_{2k+3}$. It follows that $W' \subset ( P_{2k+3} \cap \Gamma^{(4)}) $.

Note that, by Theorem \ref{Spechtbasis}, $W$ is a free abelian group with the Specht basis $S$ consisting of certain products $c_1 c_2 \dots c_m$ of commutators; these products $c_1 c_2 \dots c_m \in S$ were described in Theorem \ref{Spechtbasis}. One can easily check that a basis of $W'$ is formed by the products $c_1 c_2 \dots c_m \in S$ such that either the commutator $c_m$ is of length at least $4$ or the commutators $c_{m-1}, c_{m}$ are of length $3$. It follows that $W/W'$ is a free abelian group whose basis is formed by the images of the products $c_1 c_2 \dots c_m \in S$ such that $c_1, c_2, \dots , c_{m-1}$ are commutators of length $2$ and $c_m$ is a commutator of length $3$.

Let
\[
c_{i_1 i_2 \dots  i_{2k+3}} = [x_{i_1}, x_{i_2}] \dots [x_{i_{2k-1}}, x_{i_{2k}}] [x_{i_{2k+1}}, x_{i_{2k+2}}, x_{i_{2k+3}}].
\]
Let
\begin{multline*}
C = \Big\{ c_{i_1 \dots  i_{2k+3}} \mid \{ i_1, i_2,  \dots , i_{2k+3} \} = \{ 1,2, \dots , 2k+3 \}; \ i_1 > i_2, \dots \\
\dots , i_{2k-1} > i_{2k}; \ i_{2k+1} > i_{2k+2}, i_{2k+3}; \ i_1 < i_3 < \dots < i_{2k-1} \Big\} .
\end{multline*}
Then $C$ coincides with the set of the products $c_1 c_2 \dots c_m \in S$ above whose images form a basis of $W/W'$. Hence, we have

\begin{lemma}\label{basisW2W2'}
The set $\mathcal C = \{ c + W' \mid c \in C \}$ is a basis of $W / W'$.
\end{lemma}

Note that $W' \subset (P_{2k+3} \cap \Gamma^{(4)}) \subseteq W$. It follows immediately from Proposition \ref{generatorsaddgrp} that the group $(P_{2k+3} \cap \Gamma^{(4)})/ W'$ is generated by the images of the elements of the forms (\ref{ccc32})--(\ref{ccc222}) that belong to $P_{2k+3}$.

We claim that all elements of the form (\ref{ccc222}) belonging to $P_{2k+3}$ are contained, modulo $W'$,  in the subgroup $U$ generated by the elements of the form (\ref{ccc32}) that belong to $P_{2k+3}$. Indeed, if suffices to check that the elements
\begin{multline}\label{prodcomm}
[x_{j_1} , x_{j_2}] \dots [x_{j_{2k-7}} , x_{j_{2k-6}}] [x_{j_{2k-5}}, x_{j_{2k-4}}, x_{j_{2k-3}} ]
\\
\times \bigl( [x_{i_1}, x_{i_2}] [x_{i_3}, x_{i_4}] + [x_{i_1}, x_{i_3}] [x_{i_2}, x_{i_4}] \bigr) [x_{i_5}, x_{i_6}]
\end{multline}
belonging to $P_{2k+3}$ are contained  in $U + W'$ because all elements of the form (\ref{ccc222}) belonging to $P_{2k+3}$ are, modulo $W'$, linear combinations of the elements above.

Recall that $I_1$ is the two-sided ideal of $\mathbb Z \langle X \rangle$ generated by all polynomials of the form (\ref{cc4}), that is, by all commutators of length $4$ in $x_l$ $(l \in \mathbb N)$. Since all commutators in variables $x_l$ $(l \in \mathbb N)$ of length at least $2$ commute modulo $I_1$, we have
\begin{align*}
& [x_{j_{2k-5}}, x_{j_{2k-4}}, x_{j_{2k-3}} ] \bigl( [x_{i_1}, x_{i_2}] [x_{i_3}, x_{i_4}] + [x_{i_1}, x_{i_3}] [x_{i_2}, x_{i_4}] \bigr)
\\
\equiv & \bigl( [x_{i_1}, x_{i_2}] [x_{j_{2k-5}}, x_{j_{2k-4}}, x_{j_{2k-3}} ] + [x_{i_1}, x_{j_{2k-3}}] [x_{j_{2k-5}}, x_{j_{2k-4}}, x_{i_{2}} ] \bigr) [x_{i_3}, x_{i_4}]
\\
+ & \bigl( [x_{i_4}, x_{i_3}] [x_{j_{2k-5}}, x_{j_{2k-4}}, x_{i_{2}} ] + [x_{i_4}, x_{i_2}] [x_{j_{2k-5}}, x_{j_{2k-4}}, x_{i_{3}} ] \bigr) [x_{i_1}, x_{j_{2k-3}}]
\\
+ & \bigl( [x_{i_1}, x_{j_{2k-3}}] [x_{j_{2k-5}}, x_{j_{2k-4}}, x_{i_{3}} ] + [x_{i_1}, x_{i_{3}}] [x_{j_{2k-5}}, x_{j_{2k-4}}, x_{j_{2k-3}} ] \bigr)
\\
\times & [x_{i_2}, x_{i_4}]  \pmod{I_1} .
\end{align*}
Hence, each element of the form (\ref{prodcomm}) is, modulo $I_1$, a linear combination of elements of the form (\ref{ccc32}). Since $(P_{2k+3} \cap I_1) \subset W'$, the claim follows.

It follows that the group $(P_{2k+3} \cap \Gamma^{(4)})/ W'$ is generated by the images of the elements (\ref{ccc32}) belonging to $P_{2k+3}$, that is, by the elements $g + W'$ where
\begin{equation}
\label{generatorsmodw2'}
g = c_1 \dots c_{m} ([x_{i_1}, x_{i_2}] [ x_{i_3}, x_{i_4}, x_{i_5}] - \mbox{\rm sgn} \, (\tau )
[x_{i_{\tau (1)}}, x_{i_{\tau (2)}}] [ x_{i_{\tau (3)}}, x_{i_{\tau (4)}}, x_{i_{\tau (5)}}]),
\end{equation}
$g \in P_{2k+3}$, $c_s \in \Lambda,$ $\tau \in S_5$. It is easy to check that if for some $s, 1 \le s \le m$, the commutator $c_s$ is of length at least $3$ then the product $g$ of the form (\ref{generatorsmodw2'}) belongs to $W'$. Hence, we have
\begin{lemma}\label{generatorsP2k3G4}
The (additive) group $(P_{2k+3} \cap \Gamma^{(4)})/W'$ is generated by the elements $g +W'$ where $g$ is of the form {\rm (\ref{generatorsmodw2'})}, $m = k-1$ and $c_1, \dots , c_{(k-1)} \in \Lambda$ are commutators of length $2$.
\end{lemma}

Let $H$ be the free abelian group on a basis
\[
\Big\{ h_{i_1 i_2 \dots i_{2k+3} } \mid \{ i_1, i_2, \dots , i_{2k+3} \} = \{ 1, 2, \dots , 2k+3 \} \Big\} .
\]
Note that $H$ is isomorphic to the additive group of the group ring $\mathbb Z S_{2k+3}$ of the symmetric group $S_{2k+3}$ with an isomorphism
\[
\eta : h_{i_1 \dots i_{2k+3}} \rightarrow  \left(
\begin{array}{cccc}
1 & 2 & \dots & 2k+3 \\
i_1 & i_2 & \dots & i_{2k+3}
\end{array}
\right) \in S_{2k+3}.
\]

Define a homomorphism $\psi : H \rightarrow W/W'$ by
\[
\psi (h_{i_1 \dots i_{2k+3}}) = c_{i_1 \dots i_{2k+3}} + W'.
\]
Let $Q$ be the subgroup of $H$ generated by all elements
\begin{gather}\label{generatorQ1}
h_{i_1 \dots i_{2l-2} i_{2l-1} i_{2l} i_{2l+1} \dots i_{2k+3}} + h_{i_1 \dots  i_{2l-2} i_{2l} i_{2l-1} i_{2l+1} \dots i_{2k+3}} \ \ (1 \le l \le k+1),
\\
\label{generatorQ2}
h_{i_1 \dots i_{2l-4} i_{2l-3} i_{2l-2} i_{2l-1} i_{2l} i_{2l+1} \dots i_{2k+3}} - h_{i_1 \dots i_{2l-4} i_{2l-1} i_{2l} i_{2l-3} i_{2l-2} i_{2l+1} \dots i_{2k+3}}
\\
(2 \le l \le k) \nonumber
\end{gather}
and
\begin{equation}\label{generatorQ3}
h_{i_1 \dots i_{2k} i_{2k+1} i_{2k+2} i_{2k+3} } + h_{i_1 \dots i_{2k} i_{2k+2} i_{2k+3} i_{2k+1} } + h_{i_1 \dots i_{2k} i_{2k+3} i_{2k+1} i_{2k+2} } .
\end{equation}

\begin{lemma}
$\ker \psi = Q$.
\end{lemma}

\begin{proof}
It is clear that the elements of the forms (\ref{generatorQ1}), (\ref{generatorQ2}) and (\ref{generatorQ3}) belong to $\ker \psi$. Hence, $Q \subseteq \ker \psi$. On the other hand, one can easily check that $H/Q$ is generated by the set $\mathcal D = \{ d + Q \mid d \in D \}$ where
\begin{multline*}
D = \Big\{ h_{i_1 \dots  i_{2k+3}} \mid i_1 > i_2, \dots , i_{2k-1} > i_{2k};
\\
i_{2k+1} > i_{2k+2}, i_{2k+3}; \ i_1 < i_3 < \dots < i_{2k-1} \Big\} .
\end{multline*}
Clearly, $\mu (D) = \mathcal C$ so, by Lemma \ref{basisW2W2'}, $\mu (D)$ is a basis of $W/W'$. It follows that $\mathcal D$ is a basis of $H/Q$ and $\ker \psi = Q$, as required.
\end{proof}

Let $P$ be the subgroup of $H$ generated by all elements
\[
h_{i_1 \dots i_{2k+3}} - \mbox{sgn } (\sigma ) \ h_{ i_{ \sigma (1)} \dots i_{\sigma (2k+3)} }
\]
and
\[
h_{i_1 \dots i_{2k} i_{2k+1} i_{2k+2} i_{2k+3} } + h_{i_1 \dots i_{2k} i_{2k+2} i_{2k+3} i_{2k+1} } + h_{i_1 \dots i_{2k} i_{2k+3} i_{2k+1} i_{2k+2} }
\]
where $\{ i_1, \dots , i_{2k+3} \} = \{ 1, \dots , 2k+3 \}$, $\sigma \in S_{2k+3}$. Note that $Q \subset P$ because a generating set of $Q$ is a subset of a generating set of $P$. By Lemma \ref{generatorsP2k3G4}, the (additive) group $(P_{2k+3} \cap \Gamma^{(4)}) / W'$ is generated by the elements
\[
c_{j_1 \dots j_{2k-2} i_1 \dots i_{5}} - \mbox{sgn }(\tau ) \ c_{j_1 \dots j_{2k-2} i_{\tau (1)} \dots i_{\tau (5)} } +W'
\]
where $\{ j_1, \dots j_{2k-2}, i_1, \dots , i_{5} \} = \{ 1, \dots , 2k+3 \}$ and $\tau \in S_5$. It follows that
$\psi^{-1} \bigl( (P_{2k+3} \cap \Gamma^{(4)}) / W' \bigr)$ is generated by the elements
\[
h_{j_1 \dots j_{2k-2} i_1 \dots i_{5}} - \mbox{sgn }(\tau ) \ h_{j_1 \dots j_{2k-2} i_{\tau (1)} \dots i_{\tau (5)} }
\]
together with $\ker \psi = Q$ so
\begin{equation}\label{PGP}
\psi^{-1} \bigl( (P_{2k+3} \cap \Gamma^{(4)}) / W' \bigr) \subseteq P.
\end{equation}
(In fact, one can check that $\psi^{-1} \bigl( (P_{2k+3} \cap \Gamma^{(4)}) / W' \bigr) =  P$ but for our purpose (\ref{PGP}) is sufficient.)

Note that $v_k = c_{1 2  \dots (2k+3)} \in P_{2k+3} \cap \Gamma$ so to prove that $v_k \notin T^{(4)}$ (that is, to prove Theorem  \ref{maintheorem1}) it suffices to prove that $v_k \notin (P_{2k+3} \cap \Gamma \cap T^{(4)}) = (P_{2k+3} \cap \Gamma^{(4)})$. To prove the latter it suffices to check that $(v_k + W') = \psi (h_{1 2 \dots (2k+3)} ) \notin (P_{2k+3} \cap \Gamma^{(4)}) / W' $, that is, $h_{1 2 \dots (2k+3)} \notin \psi^{-1} \bigl( (P_{2k+3} \cap \Gamma^{(4)}) / W' \bigr) $. By (\ref{PGP}), to prove Theorem \ref{maintheorem1} it suffices to prove the following:

\begin{lemma}\label{hnotinP}
$h_{1 2 \dots (2k+3)} \notin P$.
\end{lemma}

\begin{proof}
Let $\mu : H \rightarrow \mathbb Z$ be the homomorphism of $H$ into $\mathbb Z$ defined by
\[
\mu (h_{i_1 i_2 \dots i_{2k+3}})= \mbox{sgn} (\rho )
\]
where $\rho = \left(
\begin{array}{cccc}
1 & 2 & \dots & 2k+3 \\
i_1 & i_2 & \dots & i_{2k+3}
\end{array}
\right)$. Then
\[
\mu (h_{i_1 i_2 \dots i_{2k+3}} - \mbox{sgn } (\sigma ) \ h_{i_{\sigma (1)} i_{\sigma (2)} \dots i_{\sigma (2k+3)}}) = \mbox{sgn } (\rho ) - \mbox{sgn } (\sigma ) \ \mbox{sgn } (\sigma \rho )  = 0
\]
and
\[
\mu (h_{i_1 \dots i_{2k} i_{2k+1} i_{2k+2} i_{2k+3} } + h_{i_1 \dots i_{2k} i_{2k+2} i_{2k+3} i_{2k+1} } + h_{i_1 \dots i_{2k} i_{2k+3} i_{2k+1} i_{2k+2} }) = \pm 3
\]
so $\mu (P) = 3 \, \mathbb Z$. On the other hand, $\mu (h_{1 \dots (2k+3)}) = 1 \notin 3 \, \mathbb Z $, therefore $h_{1 \dots (2k+3)} \notin P$.

This completes the proof of Lemma \ref{hnotinP} and hence of Theorem \ref{maintheorem1}.
\end{proof}

\section{Proof of Theorem 1.2 }

Recall that $T^{(4)}$ and $T^{(3,2)}$ are multihomogeneous $T$-ideals of $\mathbb Z \langle X \rangle$. Hence, by Corollary \ref{TcapGammaleft}, $T^{(4)}$  and $T^{(3,2)}$ are generated, as left ideals of $\mathbb Z \langle X \rangle$,  by the sets $\Gamma^{(4)} = T^{(4)} \cap \Gamma$  and $\Gamma^{(3,2)} = T^{(3,2)} \cap \Gamma$, respectively. By Lemma \ref{IcapGamma},
\[
T^{(4)}= \Gamma^{(4)} \oplus \bigoplus_{s \ge 1; \ i_1 \le \dots \le i_s} x_{i_1} x_{i_2} \dots x_{i_s} \Gamma^{(4)},
\]
\[
T^{(3,2)}= \Gamma^{(3,2)} \oplus \bigoplus_{s \ge 1; \ i_1 \le \dots \le i_s} x_{i_1} x_{i_2} \dots x_{i_s} \Gamma^{(3,2)}.
\]

Let
\begin{multline*}
\mathcal E' = \Big\{ [x_{i_1}, x_{i_2}] \dots [ x_{i_{2k-1}}, x_{i_{2k}}] [x_{i_{2k+1}}, x_{i_{2k+2}}, x_{i_{2k+3}}] \mid
\\
k \ge 1, \, i_1 < i_2 < \dots < i_k \Big\} .
\end{multline*}
By Corollary \ref{GammaIcapGamma}, to prove Theorem \ref{maintheorem2} it suffices to prove the following lemma.

\begin{lemma}\label{lemmamaintheorem2}
The set $\{ e' + \Gamma^{(4)} \mid e' \in \mathcal E' \}$ is a basis of the elementary abelian $3$-group $\Gamma^{(3,2)}/\Gamma^{(4)}$ over $\mathbb F_3 = \mathbb Z / 3 \, \mathbb Z$.
\end{lemma}

\begin{proof}
We claim that the set $\{ e' + \Gamma^{(4)} \mid e' \in \mathcal E' \}$ generates the (additive) group $\Gamma^{(3,2)}/\Gamma^{(4)}$. Indeed, let
\[
S = \{ [x_{i_1}, x_{i_2}, x_{i_3}][ x_{i_4}, x_{i_5}] \mid i_1 < i_2 < i_3 < i_4 < i_5 \}
\]
and let $I^{(3,2)}$ be the ideal in $\mathbb Z \langle X \rangle$ generated by $S$. We need the following lemma proved in \cite[Lemma 6.1]{Kras12}.

\begin{lemma}[see \cite{Kras12}]\label{T32}
$T^{(3,2)} = T^{(4)} + I^{(3,2)}$.
\end{lemma}

It follows from Lemma \ref{T32} that $T^{(3,2)}$ is generated as a two-sided ideal of $\mathbb Z \langle X \rangle$ by the set $\Gamma^{(4)} \cup S$. Hence, $\Gamma^{(3,2)}$ is generated as a two-sided ideal of $\Gamma$ by the same set $\Gamma^{(4)} \cup S$. In other words, $\Gamma^{(3,2)}/\Gamma^{(4)}$ is the ideal of the (commutative) algebra $\Gamma /\Gamma^{(4)}$ generated by the image of $S$. Since, for all $j_l \in \mathbb N$,
\[
[x_{j_1},x_{j_2},x_{j_3}] [x_{j_4}, x_{j_5}, x_{j_6}] \in \Gamma^{(4)} ,
\]
the (additive) group $\Gamma^{(3,2)}/\Gamma^{(4)}$ is generated by the products
\[
[x_{i_1}, x_{i_2}] \dots [x_{i_{2k-1}},x_{i_{2k}}] [x_{i_{2k+1}},x_{i_{2k+2}},x_{i_{2k+3}}] + \Gamma^{(4)} \qquad (k \ge 1, i_l \in \mathbb N).
\]
It follows easily from (\ref{cc32}) that
\begin{align*}
&[x_{i_1}, x_{i_2}] \dots [x_{i_{2k-1}},x_{i_{2k}}] [x_{i_{2k+1}},x_{i_{2k+2}},x_{i_{2k+3}}]
\\
\equiv & [x_{i_{\sigma (1)}}, x_{i_{\sigma (2)}}] \dots [x_{i_{\sigma (2k-1)}},x_{i_{\sigma (2k)}}] [x_{i_{\sigma (2k+1)}}, x_{i_{\sigma (2k+2)}}, x_{i_{\sigma (2k+3)}}] \pmod{\Gamma^{(4)}}
\end{align*}
for all permutations $\sigma \in S_{2k+3}$ so the group $\Gamma^{(3,2)}/\Gamma^{(4)}$ is generated by the elements
\[
[x_{i_1}, x_{i_2}] \dots [x_{i_{2k-1}},x_{i_{2k}}] [x_{i_{2k+1}},x_{i_{2k+2}},x_{i_{2k+3}}] + \Gamma^{(4)}
\]
such that $k \ge 1$ and $i_1 < i_2 < \dots < i_{2k+3}$. Thus, the set $\{ e' + \Gamma^{(4)} \mid e' \in \mathcal E' \}$ generates the group $\Gamma^{(3,2)}/\Gamma^{(4)}$, as claimed.

It follows easily from Theorem \ref{maintheorem1} that $e' \notin \Gamma^{(4)}$ for all $e' \in \mathcal E'$. Since $3 \cdot e' \in \Gamma^{(4)}$ $(e' \in \mathcal E')$, $\Gamma^{(3,2)}/\Gamma^{(4)}$ is a non-trivial elementary abelian $3$-group generated by the image of $\mathcal E'$.

It remains to check that the image of the set $\mathcal E'$ is linearly independent in $\Gamma^{(3,2)}/\Gamma^{(4)}$ over $\mathbb F_3 = \mathbb Z /3 \mathbb Z$. Suppose that $\alpha_1 e'_1 + \dots + \alpha_n e'_n \equiv 0 \pmod{\Gamma^{(4)}}$ $(\alpha_i \in \mathbb F_3, e_i' \in \mathcal E')$, that is,
\begin{equation}\label{lincomb}
\alpha_1 e'_1 + \dots + \alpha_n e'_n \in T^{(4)} \qquad (\alpha_i \in \mathbb F_3, e_i' \in \mathcal E').
\end{equation}
Note that the polynomials $e_1', \dots , e_n'$ belong to distinct multihomogeneous components of the ring $\mathbb Z \langle X \rangle$. Since $T^{(4)}$ is a multihomogeneous $T$-ideal, it follows from (\ref{lincomb}) that $\alpha_i e'_i \in T^{(4)}$ for each $i$, $1 \le i \le n$. Since $e'_i \notin T^{(4)}$, we have $\alpha_i =0$ for all $i$ $(\alpha_i \in \mathbb F_3)$. Thus, the set $\{ e' + \Gamma^{(4)} \mid e' \in \mathcal E' \}$  is linearly independent in $\Gamma^{(3,2)}/\Gamma^{(4)}$ over $\mathbb F_3$.

This completes the proof of Lemma \ref{lemmamaintheorem2} and of Theorem \ref{maintheorem2}.
\end{proof}

\section{Proof of Theorem 1.3 and of Corollary 1.4 }

Let $I$ be the ideal of $K \langle Y \rangle$ generated by the polynomials (\ref{c4})--(\ref{c222}). We will prove Theorem \ref{maintheorem3} by showing that $I = T^{(4)}$.

First we prove that $I \subseteq T^{(4)}$. It is clear that the elements (\ref{c4}) belong to $T^{(4)}$. Further, it is well-known (see, for instance, \cite[Theorem 3.4]{EKM09}, \cite[Lemma 1]{Gordienko07}, \cite[Lemma 2.1]{Kras12}, \cite[Lemma 2]{Latyshev65}) that, for all $a_1, \dots ,a_5 \in K \langle Y \rangle$, we have
\begin{gather}
\label{32+32T4-1}
[a_{1},a_{2},a_{3}][a_{4},a_{5}] +
[a_{1},a_{2},a_{4}][a_{3},a_{5}] \in T^{(4)},
\\
\label{32+32T4-2}
[a_{1},a_{2},a_{3}][a_{4},a_{5}] +
[a_{1},a_{4},a_{3}][a_{2},a_{5}] \in T^{(4)}.
\end{gather}
It follows that the elements (\ref{c32-1}) and (\ref{c32-2}) belong to $T^{(4)}$.

It is also well-known (see, for example, \cite[Theorem 3.2]{GL83}, \cite[Lemma 1]{Latyshev65}) that, for all $a_1, \dots ,a_6 \in K \langle Y \rangle$, we have \begin{equation}
\label{33T4}
[a_{1}, a_{2}, a_{3}] [a_{4}, a_{5}, a_{6}] \in T^{(4)}.
\end{equation}
Indeed, by (\ref{32+32T4-1}) we have
\[
[a_{1}, a_{2}, a_{3}] \bigl[ [a_{4}, a_{5}], a_{6}] + [a_{1}, a_{2}, [a_{4}, a_{5}] \bigr] [a_{3}, a_{6}] \in T^{(4)}.
\]
Since
\[
\bigl[ a_{1}, a_{2}, [a_{4}, a_{5}] \bigr] = [a_{1}, a_{2}, a_{4}, a_{5} ] - [a_{1}, a_{2}, a_{5}, a_{4} ]\in T^{(4)} ,
\]
we have $[a_{1}, a_{2}, a_{3}] [a_{4}, a_{5}, a_{6}] \in T^{(4)}$, as claimed. In particular, the elements (\ref{c33}) belong to $T^{(4)}$.

Now to prove that $I \subseteq T^{(4)}$ it remains to check that the elements (\ref{c222}) belong to $T^{(4)}$. By (\ref{32+32T4-2}), for all $a_1, \dots ,a_6 \in K \langle Y \rangle$ we have
\[
[a_{1}, a_{2}a_3, a_{4}][a_{5},a_{6}] + [a_{1},a_{5},a_{4}][a_{2} a_3, a_{6}] \in T^{(4)}
\]
so
\begin{align*}
& [a_{1}, a_{2}a_3, a_{4}][a_{5},a_{6}] + [a_{1},a_{5},a_{4}][a_{2} a_3, a_{6}]
\\
= & \bigl[ \bigl( a_2 [a_1, a_3] + [a_1, a_2] a_3 \bigr) , a_4 \bigr] [a_5, a_6] + [a_{1},a_{5},a_{4}]a_{2} [a_3, a_{6}] + [a_{1},a_{5},a_{4}][a_{2} , a_{6}] a_3
\\
= & a_2 [a_1, a_3, a_4] [a_5,a_6] + [a_2,a_4][a_1,a_3][a_5,a_6] + [a_1,a_2][a_3,a_4][a_5,a_6]
\\
+ & [a_1, a_2, a_4]a_3 [a_5,a_6] + [a_{1},a_{5},a_{4}]a_{2} [a_3, a_{6}] + [a_{1},a_{5},a_{4}][a_{2} , a_{6}] a_3 \in T^{(4)} .
\end{align*}
Note that
\begin{align*}
& a_2 [a_1, a_3, a_4] [a_5,a_6] + [a_{1},a_{5},a_{4}]a_{2} [a_3, a_{6}]
\\
= & a_2 \bigl( [a_1, a_3, a_4] [a_5,a_6] + [a_{1},a_{5},a_{4}] [a_3, a_{6}] \bigr) + [a_{1},a_{5},a_{4}, a_{2}] [a_3, a_{6}] \in T^{(4)}
\end{align*}
by (\ref{32+32T4-2}) and
\begin{align*}
& [a_1, a_2, a_4]a_3 [a_5,a_6] + [a_{1},a_{5},a_{4}][a_{2}, a_{6}] a_3
\\
= & \bigl( [a_1, a_2, a_4] [a_5,a_6] + [a_{1},a_{5},a_{4}][a_{2}, a_{6}] \bigr) a_3 -  [a_1, a_2, a_4] [a_5, a_6, a_3] \in T^{(4)}
\end{align*}
by (\ref{32+32T4-2}) and (\ref{33T4}) so $[a_2,a_4][a_1,a_3][a_5,a_6] + [a_1,a_2][a_3,a_4][a_5,a_6] \in T^{(4)}$. It follows that
\begin{multline*}
\bigl( [a_1,a_2][a_3,a_4] + [a_1,a_3][a_2,a_4] \bigr) [a_5,a_6]
\\
= [a_1,a_2][a_3,a_4][a_5,a_6] + [a_2,a_4][a_1,a_3][a_5,a_6] + \bigl[ a_1,a_3, [a_2,a_4] \bigr] [a_5,a_6] \in T^{(4)}.
\end{multline*}
In particular, the elements (\ref{c222}) belong to $T^{(4)}$.

Thus, all the generators (\ref{c4})--(\ref{c222}) of the ideal $I$ belong to $T^{(4)}$. Hence, $I \subseteq T^{(4)}$.

Now we prove that $T^{(4)} \subseteq I$. Since the ideal $T^{(4)}$ is generated by the polynomials $[a_1, a_2, a_3, a_4]$ $( a_i \in K \langle Y \rangle )$, it suffices to check that, for all $a_i \in K \langle Y \rangle $, the polynomial
\begin{equation}\label{4inI}
[a_1, a_2, a_3, a_4]
\end{equation}
belongs to $I$. Clearly, one can assume that all $a_i$ are monomials.

In order to prove that each polynomial of the form (\ref{4inI}) belongs to $I$ we will prove that, for all monomials $a_i \in K \langle Y \rangle$, the following polynomials belong to $I$ as well:
\begin{gather}\label{222inI}
\bigl( [a_1, a_2][a_3, a_4] + [a_1, a_3][a_2, a_4] \bigr) [a_5, a_6];
\\
\label{33inI} [a_1,a_2,a_3] [a_4,a_5,a_6] ;
\\
\label{32inI-1} [a_1, a_2, a_3] [a_4, a_5] + [a_1, a_2, a_4] [a_3, a_5] ;
\\
\label{32inI-2} [a_1, a_2, a_3] [a_4, a_5] + [a_1, a_4, a_3] [a_2, a_5] .
\end{gather}
The proof is by induction on the degree $m = \deg f$ of a polynomial $f$ that is of one of the forms (\ref{4inI})--(\ref{32inI-2}). It is clear that $m \ge 4$. If $m = 4$ then $f$ is of the form (\ref{4inI}) with all monomials $a_i$ of degree $1$. Hence, $f$ is of the form (\ref{c4}) so $f \in I$. This establishes the base of the induction.

To prove the induction step suppose that $m = \deg f >4$ and that all polynomials of the forms (\ref{4inI})--(\ref{32inI-2}) of degree less than $m$ belong to $I$.

\textbf{Case 1.} Suppose that $f$ is of the form (\ref{222inI}),
\[
f = f(a_1, \dots , a_6) = \bigl( [a_1, a_2][a_3, a_4] + [a_1, a_3][a_2, a_4] \bigr) [a_5, a_6].
\]
If $\mbox{deg } f = 6$ then $f$ is of the form (\ref{c222}) so $f \in I$. If $\mbox{deg }f > 6$ then, for some $i$, $1 \le i \le 6$, we have $a_i = a'_i a''_i$ where $\mbox{deg } a'_i, \mbox{deg } a''_i < \mbox{deg } a_i$. We claim that to check that $f \in I$ one may assume without loss of generality that $i = 1$ or $i = 6$.

Indeed, we have $f(a_1, \dots , a_4, a_5, a_6) = - f(a_1, \dots , a_4, a_6, a_5)$. Further, by the induction hypothesis, $\bigl[ [a_i, a_j], [a_k, a_l] \bigr] \in I$ for all distinct $i,j,k,l$, $1 \le i,j,k,l \le 6$. It follows that
\begin{multline*}
f(a_1, a_2, a_3, a_4, a_5, a_6 ) \equiv f(a_2, a_1, a_4, a_3, a_5, a_6 )
\\
\equiv f(a_3, a_1, a_4, a_2, a_5, a_6) \equiv f(a_4, a_2, a_3, a_1, a_5, a_6) \pmod{I} .
\end{multline*}
The claim follows.

Suppose that $a_6 = a'_6 a''_6$. We have
\begin{align*}
f = & \bigl( [a_1, a_2][a_3, a_4] + [a_1, a_3][a_2, a_4] \bigr) [a_5, a'_6 a''_6]
\\
= & \bigl( [a_1, a_2][a_3, a_4] + [a_1, a_3][a_2, a_4] \bigr) \bigl( a'_6 [a_5, a''_6]
+ [a_5, a'_6] a''_6 \bigr)
\\
= & \bigl( [a_1, a_2][a_3, a_4] + [a_1, a_3][a_2, a_4] \bigr) \bigl( [a_5, a''_6] a'_6
- [a_5, a''_6, a'_6] + [a_5, a'_6] a''_6 \bigr) .
\end{align*}
By the induction hypothesis, $\bigl( [a_1,a_2][a_3,a_4] + [a_1,a_3] [a_2, a_4] \bigr) [a_5, b] \in I$ where $b \in \{ a'_6, a''_6 \}$. It follows that
\[
\bigl( [a_1,a_2][a_3,a_4] + [a_1,a_3][a_2,a_4] \bigr) \bigl( [a_5, a''_6] a'_6  + [a_5, a'_6] a''_6 \bigr) \in I.
\]
On the other hand, by the induction hypothesis, we have
\begin{align*}
& [a_3,a_4] [a_5, a''_6, a'_6] + [a'_6,a_4] [a_5, a''_6, a_3] \in I,
\\
& [a_1,a_2] [a_5, a''_6, a_3] + [a_1,a_3] [a_5, a''_6, a_2] \in I,
\\
& [a'_6,a_4] [a_5, a''_6, a_2] + [a_2,a_4] [a_5, a''_6, a'_6] \in I
\end{align*}
so
\begin{align*}
& \bigl( [a_1,a_2][a_3,a_4] + [a_1,a_3][a_2,a_4] \bigr) [a_5, a''_6, a'_6]
\\
= & [a_1,a_2]\bigl( [a_3,a_4] [a_5, a''_6, a'_6] + [a'_6,a_4] [a_5, a''_6, a_3] \bigr)
\\
- & [a'_6, a_4] \bigl( [a_1,a_2] [a_5, a''_6, a_3] + [a_1,a_3] [a_5, a''_6, a_2] \bigr)
\\
+ & [a_1,a_3] \bigl( [a'_6,a_4] [a_5, a''_6, a_2] + [a_2,a_4] [a_5, a''_6, a'_6] \bigr) \in I .
\end{align*}
It follows that $f \in I$ if $a_6 = a'_6 a''_6$.

Suppose that $a_1 = a'_1 a''_1$. We have
\begin{align*}
f = & \bigl( [a'_1 a''_1, a_2][a_3, a_4] + [a'_1 a''_1, a_3][a_2, a_4] \bigr) [a_5, a_6]
\\
= & \Big( a'_1 \bigl( [a''_1, a_2][a_3, a_4] + [a''_1, a_3][a_2, a_4] \bigr) + [a'_1, a_2] a''_1 [a_3, a_4] + [a'_1, a_3] a''_1 [a_2, a_4] \Big)
\\
\times & [a_5, a_6] = \Big( a'_1 \bigl( [a''_1, a_2][a_3, a_4] + [a''_1, a_3][a_2, a_4] \bigr) + a''_1 \bigl( [a'_1, a_2][a_3, a_4]
\\
+ & [a'_1, a_3][a_2, a_4] \bigr) +  [a'_1, a_2, a''_1][a_3, a_4] + [a'_1, a_3, a''_1][a_2, a_4] \Big) [a_5, a_6].
\end{align*}
By the induction hypothesis, $([b, a_2][a_3, a_4] + [b, a_3][a_2, a_4]) [a_5, a_6] \in I$ where $b \in \{ a'_1, a''_1 \}$. It follows that
\begin{multline*}
\Big( a'_1 \bigl( [a''_1, a_2][a_3, a_4] + [a''_1, a_3][a_2, a_4] \bigr)
\\
+ a''_1 \bigl( [a'_1, a_2][a_3, a_4] + [a'_1, a_3][a_2, a_4] \bigr) \Big) [a_5, a_6] \in I.
\end{multline*}
On the other hand, by the induction hypothesis,
\[
[a'_1, a_2, a''_1][a_3, a_4] + [a'_1, a_3, a''_1][a_2, a_4] \in I
\]
so
\[
\bigl( [a'_1, a_2, a''_1][a_3, a_4] + [a'_1, a_3, a''_1][a_2, a_4] \bigr) [a_5, a_6] \in I.
\]
It follows that $f \in I$ if $ a_1 = a'_1 a''_1$.

Thus, if $f$ is a polynomial of the form (\ref{222inI}) of degree $m$ then $f \in I$.

\textbf{Case 2.} Suppose that $f$ is of the form (\ref{33inI}),
\[
f = f(a_1, \dots , a_6) = [a_1, a_2, a_3] [a_4, a_5, a_6] .
\]
If $\mbox{deg }f = 6$ then $f$ is of the form (\ref{c33}) so $f \in I$. If $\mbox{deg }f > 6$ then, for some $i$, $1 \le i \le 6$, we have $a_i = a'_i a''_i$ where $\mbox{deg } a'_i, \mbox{deg } a''_i < \mbox{deg } a_i$. We claim that to check that $f \in I$ one may assume without loss of generality that $i=1$ or $i = 3$.

Indeed, $f(a_1, a_2, a_3, \dots ) = - f (a_2, a_1, a_3, \dots )$. Further, by the induction hypothesis, $[a_1, a_2, a_3, a_i] \in I$ for $i= 4, 5, 6$; it follows that $[a_1, a_2, a_3]$ commutes with $[a_4,a_5,a_6]$ modulo $I$ so
\[
f(a_1, a_2, a_3, a_4, a_5, a_6) \equiv f(a_4,a_5,a_6,a_1,a_2,a_3) \pmod{I}.
\]
The claim follows.

Suppose that $a_3 = a'_3 a''_3$. Then
\begin{align*}
f = & [a_1, a_2, a'_3 a''_3] [a_4, a_5, a_6] =  a'_3 [a_1, a_2, a''_3] [a_4, a_5, a_6] + [a_1, a_2, a'_3 ] a''_3 [a_4, a_5, a_6]
\\
= & a'_3 [a_1, a_2, a''_3] [a_4, a_5, a_6] +  a''_3 [a_1, a_2, a'_3 ] [a_4, a_5, a_6] + [a_1, a_2, a'_3 , a''_3] [a_4, a_5, a_6].
\end{align*}
%
%
By the induction hypothesis, $[a_1, a_2, b] [a_4, a_5, a_6] \in I$
where $b \in \{ a'_3, a''_3 \}$ and $[a_1, a_2, a'_3 , a''_3] \in
I$ so in this case $f \in I$, as required.

Suppose that $a_1 = a'_1 a''_1$. We have
\begin{align*}
f = & [a'_1 a''_1, a_2, a_3] [a_4, a_5, a_6] = \bigl[ \bigl( a'_1 [a''_1, a_2] + [a'_1, a_2] a''_1 \bigr) , a_3\bigr]  [a_4, a_5, a_6]
\\
= & \bigl( a'_1 [a''_1, a_2, a_3] + [a_1', a_3] [a''_1, a_2] + [a'_1 , a_2]
[a''_1, a_3] + [a'_1, a_2, a_3] a''_1 \bigr) [a_4, a_5, a_6] \\
= & \bigl( a'_1 [a''_1, a_2, a_3] + [a'_1, a_3] [a''_1, a_2] + [a'_1 ,
a_2] [a''_1, a_3] + a''_1 [a'_1, a_2, a_3]
\\
+ & [a'_1, a_2, a_3, a''_1] \bigr) [a_4, a_5, a_6].
\end{align*}
By the induction hypothesis, $[b, a_2, a_3] [a_4, a_5, a_6] \in I$
where $b \in \{ a'_1, a''_1 \}$ and $[a'_1, a_2, a_3, a''_1] \in
I$. It follows that
\[
\bigl( a'_1 [a''_1, a_2, a_3] + a''_1 [a'_1, a_2, a_3] + [a'_1, a_2,
a_3, a''_1] \bigr) [a_4, a_5, a_6] \in I.
\]
On the other hand, by the induction hypothesis,
\begin{align*}
& [a''_1, a_2] [a_4, a_5, a_6] + [a_6, a_2] [a_4, a_5, a''_1] \in I ,
\\
& [a'_1, a_3] [a_4, a_5, a''_1] + [a''_1, a_3] [a_4, a_5, a'_1] \in I,
\\
& [a_6, a_2] [a_4, a_5, a'_1] + [a'_1, a_2] [a_4, a_5, a_6] \in I
\end{align*}
so
\begin{align*}
& \bigl( [a'_1, a_3] [a''_1, a_2] + [a'_1, a_2] [a''_1, a_3] \bigr) [a_4, a_5, a_6]
\\
= & [a'_1, a_3] \bigl( [a''_1, a_2] [a_4, a_5, a_6] + [a_6, a_2] [a_4, a_5, a''_1] \bigr)
\\
- & [a_6, a_2] \bigl( [a'_1, a_3] [a_4, a_5, a''_1] + [a''_1, a_3] [a_4, a_5, a'_1] \bigr)
\\
+ & [a''_1, a_3] \bigl( [a_6, a_2] [a_4, a_5, a'_1] + [a'_1, a_2] [a_4, a_5, a_6] \bigr) \in I .
\end{align*}
It follows that $f \in I$ if $a_1 = a'_1 a''_1$.

Thus, if $f$ is a polynomial of the form (\ref{33inI}) of degree $m$ then $f \in I$.

\textbf{Case 3.} Suppose that $f$ is of the form (\ref{32inI-1}). If $\mbox{deg
}f = 5$ then $f$ is of the form (\ref{c32-1}) so $f \in I$. If
$\mbox{deg }f > 5$ then, for some $i$, $1 \le i \le 5$, we have
$a_i = a'_i a''_i$ where $\mbox{deg } a'_i, \mbox{deg } a''_i <
\mbox{deg } a_i$.

Suppose first that $a_5 = a'_5 a''_5$. Then
%
%
\begin{align*}
f = & [a_1, a_2, a_3] [a_4, a'_5 a''_5] + [a_1, a_2, a_4] [a_3, a'_5 a''_5]
\\
=  & [a_1, a_2, a_3] a'_5 [a_4, a''_5]  +  [a_1, a_2, a_4] a'_5 [a_3, a''_5]
\\
+ & \bigl( [a_1, a_2, a_3] [a_4, a'_5] +  [a_1, a_2, a_4] [a_3, a'_5] \bigr) a''_5
\\
= & a'_5 \bigl( [a_1, a_2, a_3] [a_4, a''_5]  +  [a_1, a_2, a_4] [a_3, a''_5] \bigr)
\\
+ & \bigl( [a_1, a_2, a_3] [a_4, a'_5] +  [a_1, a_2, a_4] [a_3, a'_5] \bigr) a''_5
\\
+ & [a_1, a_2, a_3, a'_5] [a_4, a''_5]  +  [a_1, a_2, a_4, a'_5] [a_3, a''_5]
\end{align*}
%
%
so in this case $f \in I$ by the induction hypothesis.

Suppose that $a_4 = a'_4 a''_4$. Then
\begin{align*}
f = & [a_1, a_2, a_3] [a'_4 a''_4, a_5] + [a_1, a_2, a'_4 a''_4] [a_3, a_5]
\\
= & [a_1, a_2, a_3] a'_4 [a''_4, a_5] + a'_4 [a_1, a_2,  a''_4] [a_3, a_5]
\\
+ & [a_1, a_2, a_3] [a'_4 , a_5] a''_4 + [a_1, a_2, a'_4 ] a''_4 [a_3, a_5]
\\
= & a_4' \bigl( [a_1, a_2, a_3] [a''_4, a_5] + [a_1, a_2,  a''_4] [a_3, a_5] \bigr) +  [a_1, a_2, a_3, a'_4] [a''_4, a_5]
\\
+ & \bigl( [a_1, a_2, a_3] [a'_4 , a_5] + [a_1, a_2, a'_4 ]  [a_3, a_5] \bigr) a''_4 - [a_1, a_2, a'_4 ]  [a_3, a_5, a''_4].
\end{align*}
%
%
Since each polynomial of the form (\ref{33inI}) of degree $m$ belongs to $I$, we have $[a_1, a_2, a'_4 ]  [a_3, a_5, a''_4] \in I$. On the other hand, by the induction hypothesis,
\begin{align*}
& [a_1, a_2, a_3] [a''_4, a_5] + [a_1, a_2,  a''_4] [a_3, a_5]  \in I ,
\\
& [a_1, a_2, a_3] [a'_4 , a_5] + [a_1, a_2, a'_4 ]  [a_3, a_5] \in I
\end{align*}
and $[a_1, a_2, a_3, a'_4] \in I$. It follows that in this case $f \in I$. Similarly, $f \in I$ if $a_3 = a'_3 a''_3$.

Suppose that $a_1 = a_1' a_1''$. Then
\begin{align*}
f = & [a_1' a_1'', a_2, a_3][a_4, a_5] + [a_1' a_1'', a_2, a_4][a_3, a_5]
\\
= & \bigl[ \bigl( a_1' [a_1'', a_2] + [a_1' , a_2] a_1'' \bigr) , a_3\bigr] [a_4, a_5] + \bigl[ \bigl( a_1' [a_1'', a_2] + [a_1' , a_2] a_1'' \bigr) , a_4 \bigr] [a_3, a_5]
\\
= & a_1' [a_1'', a_2, a_3] [a_4, a_5] + [a_1', a_3] [a_1'', a_2][a_4, a_5] + [a_1', a_2] [a_1'', a_3][a_4, a_5]
\\
+ & [a_1', a_2, a_3] a_1'' [a_4, a_5] + a_1' [a_1'', a_2, a_4] [a_3, a_5] + [a_1', a_4] [a_1'', a_2][a_3, a_5]
\\
+ & [a_1' , a_2] [a_1'', a_4][a_3, a_5]  + [a_1' , a_2, a_4] a_1'' [a_3, a_5]
\\
= & a_1' \bigl( [a_1'', a_2, a_3] [a_4, a_5] + [a_1'', a_2, a_4] [a_3, a_5] \bigr)
\\
+ & a_1'' \bigl( [a_1', a_2, a_3]  [a_4, a_5] + [a_1' , a_2, a_4] [a_3, a_5] \bigr) + [a_1', a_2, a_3, a_1''] [a_4, a_5]
\\
+ & [a_1' , a_2, a_4, a_1''] [a_3, a_5] - \bigl( [a_1', a_3] [a_2, a_1''] + [a_1', a_2] [a_3, a_1'']\bigr) [a_4, a_5]
\\
- & \bigl( [a_1', a_4] [a_2, a_1''] + [a_1' , a_2] [a_4, a_1''] \bigr) [a_3, a_5] .
\end{align*}
Note that
\begin{align*}
& \bigl( [a_1', a_3] [a_2, a_1''] + [a_1', a_2] [a_3, a_1'']\bigr) [a_4, a_5] \in I ,
\\
& \bigl( [a_1', a_4] [a_2, a_1''] + [a_1' , a_2] [a_4, a_1''] \bigr) [a_3, a_5] \in I
\end{align*}
because the polynomials of the form (\ref{222inI}) of degree $m$ belong to $I$. On the other hand, by the induction hypothesis,
\begin{align*}
& [a_1'', a_2, a_3] [a_4, a_5] + [a_1'', a_2, a_4] [a_3, a_5] \in I,
\\
& [a_1', a_2, a_3]  [a_4, a_5] + [a_1' , a_2, a_4] [a_3, a_5] \in I
\end{align*}
and $[a_1' , a_2, a_4, a_1''] \in I$ so in this case $f \in I$. Similarly, $f \in I$ if $a_2 = a'_2 a''_2$.

Thus, if $f$ is a polynomial of the form (\ref{32inI-1}) of degree $m$ then $f \in I$.

%
%

\textbf{Case 4.} Suppose that $f$ is of the form (\ref{32inI-2}). If $\mbox{deg }f = 5$ then $f$ is of the form (\ref{c32-2}) so $f \in I$. If $\mbox{deg }f > 5$ then, for some $i$, $1 \le i \le 5$, we have $a_i = a'_i a''_i$ where $\mbox{deg } a'_i, \mbox{deg } a''_i < \mbox{deg } a_i$.

Suppose first that $a_5 = a'_5 a''_5$. Then
\begin{align*}
f = & [a_1, a_2, a_3] [a_4, a_5' a_5''] + [a_1, a_4, a_3] [a_2, a_5' a_5'']
\\
= & [a_1, a_2, a_3] a_5' [a_4, a_5''] + [a_1, a_2, a_3] [a_4, a_5'] a_5''
\\
+ & [a_1, a_4, a_3] a_5' [a_2,  a_5''] +  [a_1, a_4, a_3] [a_2, a_5'] a_5''
\\
= & a_5' \bigl( [a_1, a_2, a_3] [a_4, a_5''] +  [a_1, a_4, a_3] [a_2,  a_5''] \bigr) + \bigl( [a_1, a_2, a_3] [a_4, a_5']
\\
+  & [a_1, a_4, a_3] [a_2, a_5'] \bigr) a_5'' + [a_1, a_2, a_3, a_5'] [a_4, a_5''] + [a_1, a_4, a_3, a_5'] [a_2,  a_5'']
\end{align*}
so in this case $f \in I$ by the induction hypothesis.

Suppose that $a_4 = a'_4 a''_4$. Then
\begin{align*}
f = & [a_1, a_2, a_3] [a_4' a_4'', a_5] + [a_1, a_4' a_4'', a_3] [a_2, a_5]
\\
= & [a_1, a_2, a_3] a_4' [a_4'', a_5] + [a_1, a_2, a_3] [a_4' , a_5] a_4''
\\
+ & \bigl[ \bigl( [a_1, a_4'] a_4'' + a_4' [a_1,  a_4''] \bigr) , a_3 \bigr] [a_2, a_5]
\\
= & [a_1, a_2, a_3] a_4' [a_4'', a_5] + [a_1, a_2, a_3] [a_4' , a_5] a_4'' + [a_1, a_4'][a_4'', a_3][a_2, a_5]
\\
+ & [a_1, a_4', a_3]a_4'' [a_2, a_5] + a_4' [a_1,  a_4'', a_3] [a_2, a_5] + [a_4', a_3][a_1, a_4''] [a_2, a_5]
\\
= & a_4' \bigl( [a_1, a_2, a_3] [a_4'', a_5] + [a_1,  a_4'', a_3] [a_2, a_5] \bigr) + [a_1, a_2, a_3, a_4'] [a_4'', a_5]
\\
+ & \bigl( [a_1, a_2, a_3] [a_4' , a_5]  + [a_1, a_4', a_3] [a_2, a_5] \bigr) a_4'' - [a_1, a_4', a_3] [a_2, a_5, a_4'']
\\
+ & ([a_1, a_4'][a_4'', a_3] + [a_1, a_4''] [a_4', a_3]) [a_2, a_5] + \bigl[ a_4', a_3,[a_1, a_4''] \bigr] [a_2, a_5].
\end{align*}
It follows that in this case $f \in I$ by the induction hypothesis and because the polynomials of the forms  (\ref{222inI})--(\ref{33inI}) of degree $m$ belong to $I$. Similarly, $f \in I$ if $a_2 = a'_2 a''_2$.

Suppose that $a_3 = a'_3 a''_3$. We have
\begin{align*}
f = & [a_1, a_2, a_3'a_3''] [a_4, a_5] + [a_1, a_4, a_3'a_3''] [a_2, a_5]
\\
= & a_3' [a_1, a_2, a_3''] [a_4, a_5] + [a_1, a_2, a_3'] a_3'' [a_4, a_5]
\\
+ & a_3' [a_1, a_4, a_3''] [a_2, a_5] + [a_1, a_4, a_3'] a_3'' [a_2, a_5]
\\
= & a_3' \bigl( [a_1, a_2, a_3''] [a_4, a_5] + [a_1, a_4, a_3''] [a_2, a_5] \bigr)
\\
+ & \bigl( [a_1, a_2, a_3'] [a_4, a_5] +  [a_1, a_4, a_3']  [a_2, a_5] \bigr) a_3''
\\
- & [a_1, a_2, a_3']  [a_4, a_5, a_3''] - [a_1, a_4, a_3']  [a_2, a_5, a_3'']
\end{align*}
so in this case $f \in I$ as above.

Suppose that $a_1 = a'_1 a''_1$. Then
\begin{align*}
f = & [a_1'a_1'', a_2, a_3] [a_4, a_5] + [a_1'a_1'', a_4, a_3] [a_2, a_5]
\\
= & \bigl[ \bigl( a_1'[a_1'', a_2] + [a_1', a_2] a_1'' \bigr) , a_3 \bigr] [a_4, a_5] + \bigl[ \bigl( a_1'[a_1'', 
\\
= & a_1'[a_1'', a_2, a_3][a_4, a_5] + [a_1', a_3][a_1'', a_2][a_4, a_5] + [a_1', a_2][a_1'', a_3][a_4, a_5]
\\
+ & [a_1', a_2, a_3] a_1'' [a_4, a_5] + a_1'[a_1'', a_4, a_3][a_2, a_5] + [a_1', a_3] [a_1'', a_4][a_2, a_5]
\\
+ & [a_1', a_4][a_1'',a_3][a_2, a_5] + [a_1', a_4, a_3] a_1'' [a_2,a_5]
\\
= & a_1' \bigl( [a_1'', a_2, a_3][a_4, a_5] + [a_1'', a_4, a_3][a_2, a_5] \bigr) +  a_1'' \bigl( [a_1', a_2, a_3]  [a_4, a_5]
\\
+ & [a_1', a_4, a_3] [a_2,a_5] \bigr) + [a_1', a_2, a_3, a_1''] [a_4, a_5] + [a_1', a_4, a_3, a_1''] [a_2,a_5]
\\
+ & [a_1', a_3]([a_1'', a_2][a_4, a_5] + [a_1'', a_4][a_2, a_5]) + \bigl( [a_1', a_2][a_4, a_5] + [a_1', a_4][a_2, a_5] \bigr)
\\
\times & [a_1'', a_3] + [a_1', a_2]\bigl[ a_1'', a_3, [a_4, a_5] \bigr] +  [a_1', a_4] \bigl[ a_1'',a_3, [a_2, a_5] \bigr]
\end{align*}
so in this case $f \in I$ by the induction hypothesis and because the polynomials of the form (\ref{222inI}) of 

Thus, if $f$ is a polynomial of the form (\ref{32inI-2}) of degree $m$ then $f \in I$.

\textbf{Case 5.} Finally, suppose that $f$ is of the form (\ref{4inI}). Since $\mbox{deg }f > 4$, we have $a_i = a'_i a''_i$ for some $i$, $1 \le i \le 4$ where $\mbox{deg } a'_i, \mbox{deg } a''_i < \mbox{deg } a_i$.

Suppose that $a_4 = a'_4 a''_4$. Then
\[
f = [a_1, a_2, a_3, a'_4 a''_4] = a'_4 [a_1, a_2, a_3, a''_4] + [a_1,
a_2, a_3, a'_4 ] a''_4
\]
so, by the induction hypothesis, $f \in I$.

Now suppose that $a_3 = a'_3 a''_3$. Then
\begin{align*}
f= & [a_1, a_2, a'_3 a''_3, a_4] = [a'_3 [a_1, a_2, a''_3] + [a_1, a_2, a'_3] a''_3 , a_4]
\\
= & a'_3 [a_1, a_2, a''_3, a_4] + [a'_3, a_4] [a_1, a_2, a''_3] + [a_1, a_2, a'_3] [a''_3 , a_4] + [a_1, a_2, a'_3, a_4] a''_3
\\
= & a'_3 [a_1, a_2, a''_3, a_4] + [a_1, a_2, a''_3][a'_3, a_4]  + [a_1, a_2, a'_3] [a''_3 , a_4]
\\
+ & [a_1, a_2, a'_3, a_4] a''_3 + \bigl[ a'_3, a_4, [a_1, a_2, a''_3] \bigr] .
\end{align*}
By the induction hypothesis, $[a_1, a_2, b, a_4] \in I$ if $b \in \{ a'_3, a''_3 \}$ and
\begin{multline*}
\bigl[ a'_3, a_4, [a_1, a_2, a''_3] \bigr] = - \bigl[ [a_1, a_2, a''_3], [a'_3,
a_4] \bigr]
\\
= - [a_1, a_2, a''_3, a'_3, a_4] +  [a_1, a_2, a''_3, a_4, a'_3] \in I.
\end{multline*}
On the other hand,  $[a_1, a_2, a''_3][a'_3, a_4]  + [a_1, a_2, a'_3] [a''_3 , a_4] \in I$ because each polynomial of the form (\ref{32inI-1}) of degree $m$ belongs to $I$. Hence, in this case $f \in I$.

Finally suppose that either $a_1 = a_1' a_1''$ or $a_2 = a'_2 a''_2$. It is clear that without loss of generality we may assume $a_2 = a'_2 a''_2$. Then
\begin{align*}
f = & [a_1, a'_2 a''_2, a_3, a_4] = \bigl[ \bigl( a'_2 [a_1, a''_2] + [a_1, a'_2]
a''_2 \bigr) , a_3, a_4 \bigr]
\\
= & \bigl[ \bigl( a'_2 [a_1, a''_2, a_3] + [a'_2, a_3] [a_1, a''_2] + [a_1, a'_2]
[a''_2, a_3] + [a_1, a'_2, a_3] a''_2 \bigr) , a_4 \bigr]
\\
= & a'_2 [a_1, a''_2, a_3, a_4] + [a'_2, a_4] [a_1, a''_2, a_3] +
[a'_2, a_3] [a_1, a''_2, a_4]
\\
+ & [a'_2, a_3, a_4] [a_1, a''_2] + [a_1, a'_2] [a''_2, a_3, a_4] +
[a_1, a'_2, a_4] [a''_2, a_3]
\\
+ & [a_1, a'_2, a_3] [a''_2, a_4] + [a_1, a'_2, a_3, a_4] a''_2.
\end{align*}
We have $ a'_2 [a_1, a''_2, a_3, a_4], [a_1, a'_2, a_3, a_4] a''_2 \in I $ by the induction hypothesis and
\[
[a_1, a'_2, a_4] [a''_2, a_3] + [a_1, a'_2, a_3] [a''_2, a_4] = -
\bigl( [a_1, a'_2, a_4] [a_3, a''_2] + [a_1, a'_2, a_3] [a_4, a''_2] \bigr)
\in I
\]
because each polynomial of the form (\ref{32inI-1}) of degree $m$ belong to $I$. Further,
\begin{align*}
& [a'_2, a_4] [a_1, a''_2, a_3] + [a'_2, a_3] [a_1, a''_2, a_4]
\\
= & [a_1, a''_2, a_3] [a'_2, a_4] +  [a_1, a''_2, a_4] [a'_2, a_3] -
\bigl[ [a_1, a''_2, a_3], [a'_2, a_4] \bigr] -  \bigl[ [a_1, a''_2, a_4], [a'_2,
a_3] \bigr]
\end{align*}
where $[a_1, a''_2, a_3] [a'_2, a_4] +  [a_1, a''_2, a_4] [a'_2,
a_3] \in I$ as above,
\[
\bigl[ [a_1, a''_2, a_3], [a'_2, a_4] \bigr] = [a_1, a''_2, a_3, a'_2, a_4] -
[a_1, a''_2, a_3, a_4, a'_2] \in I
\]
by the induction hypothesis and, similarly, $\bigl[ [a_1, a''_2, a_4],
[a'_2, a_3] \bigr] \in I$. Finally,
\begin{align*}
& [a'_2, a_3, a_4] [a_1, a''_2] + [a_1, a'_2] [a''_2, a_3, a_4]
\\
= & [a_3, a'_2, a_4] [a''_2, a_1] + [a_3, a''_2, a_4] [a'_2, a_1] -
\bigl[ [a_3, a''_2, a_4], [a'_2, a_1] \bigr]
\end{align*}
where $[a_3, a'_2, a_4] [a''_2, a_1] + [a_3, a''_2, a_4] [a'_2,
a_1] \in I$ and $\bigl[ [a_3, a''_2, a_4], [a'_2,
a_1] \bigr] \in I$ as above. It follows that $f =[a_1, a'_2 a''_2, a_3,
a_4] \in I$, as required.

Thus, $[a_1, a_2, a_3, a_4] \in I$ for all monomials $a_i \in K \langle X \rangle$ such that $[a_1, a_2, a_3, a_4]$ is of degree $m$. This establishes the induction step and proves that $T^{(4)} \subseteq I$.

The proof of Theorem \ref{maintheorem3} is completed.

\medskip
Now we prove Corollary \ref{corollarytomaintheorem3}.

Let $I$ be the two-sided ideal of $K \langle Y \rangle$ generated by the polynomials (\ref{c4}), (\ref{c33}), (\ref{c32sigma}) and (\ref{c222sigma}). We will prove that $I = T^{(4)}$. Since the ideal $T^{(4)}$ is generated by the polynomials (\ref{c4}), (\ref{c33}), (\ref{c32-1}), (\ref{c32-2}) and (\ref{c222}) where polynomials (\ref{c32-1}) and (\ref{c32-2}) are of the form (\ref{c32sigma}) and polynomials (\ref{c222}) are of the form (\ref{c222sigma}), we have $T^{(4)} \subseteq I$. To prove that $I \subseteq T^{(4)}$ it suffices to check that all polynomials (\ref{c32sigma}) and (\ref{c222sigma}) belong to $T^{(4)}$.

The following lemma is well-known (see, for instance, \cite{EKM09}, \cite[Lemma 1]{Gordienko07}, \cite[Lemma 2.3]{Kras12}\cite[Lemma 1]{Volichenko78}).

\begin{lemma}\label{alternating}
For all $a_1, \dots ,a_5 \in \mathbb Z \langle X \rangle$ and all $\sigma \in S_5$, we have
\begin{equation}\label{permutation2}
[a_{1}, a_{2}] [a_{3}, a_{4}, a_{5}] \equiv  \mbox{sgn} (\sigma )
[a_{\sigma (1)}, a_{\sigma (2)}] [a_{\sigma (3)}, a_{\sigma (4)},
a_{\sigma (5)}] \pmod{T^{(4)}} .
\end{equation}
\end{lemma}

\begin{proof}
It is clear that the congruence (\ref{permutation2}) holds if $\sigma = (12)$ or $\sigma = (34)$. If $\sigma = (25)$ or $\sigma = (24)$ then, by Theorem \ref{maintheorem3}, the congruence (\ref{permutation2}) holds as well. Since the transpositions $(12)$, $(34)$, $(25)$ and $(24)$ generate the entire group $S_5$ of the permutations of the set $\{ 1,2,3,4,5 \}$, the result follows.
\end{proof}

It follows from Lemma \ref{alternating} that the polynomials (\ref{c32sigma}) belong to $T^{(4)}$. One can prove in a similar way that the polynomials (\ref{c222sigma}) belong to $T^{(4)}$ as well. Thus, $I \subseteq T^{(4)}$, as required.

The proof of Corollary \ref{corollarytomaintheorem3} is completed.


\begin{thebibliography}{99}

\bibitem{AS01}
Bernhard Amberg, Yaroslav Sysak, \textit{Associative rings whose adjoint semigroup is locally nilpotent}, Archiv der Mathematik (Basel) \textbf{76} (2001), 426--435.

\bibitem{AJ10}
Noah Arbesfeld, David Jordan, \textit{New results on the lower central series quotients of a free associative algebra}, Journal of Algebra \textbf{323} (2010), 1813--1825. arXiv:0902.4899

\bibitem{BB11}
Martina Balagovi\'c, Anirudha Balasubramanian,  \textit{On the lower central series quotients of a graded associative algebra}, Journal of Algebra \textbf{328} (2011), 287--300.

\bibitem{BJ10}
Asilata Bapat, David Jordan, \textit{Lower central series of free algebras in symmetric tensor categories}, Journal of Algebra \textbf{373} (2013), 299--311. arXiv:1001.1375

\bibitem{BEJKL12}
Surya Bhupatiraju, Pavel Etingof, David Jordan, William Kuszmaul and Jason Li, \textit{Lower central series of a free associative algebra over the integers and finite fields}, Journal of Algebra \textbf{372} (2012), 251--274. arXiv:1203.1893

\bibitem{CK13}
Eudes Antonio da Costa, Alexei Krasilnikov, \textit{Relations in universal Lie nilpotent associative algebras of class $4$}, arXiv:1306.4294

\bibitem{CFZ13}
Katherine Cordwell, Teng Fei, Kathleen Zhou, \textit{On lower central ceries quotients of finitely generated algebras over $\mathbb Z$}, arXiv:1309.1237

\bibitem{DE08}
Galyna Dobrovolska, Pavel Etingof, \textit{An upper bound for the lower central series quotients of a free associative algebra}, International Mathematics Research Notices \textbf{2008}, no. 12, Art. ID rnn039, 10 pp.  arXiv:0801.1997

\bibitem{DKM08}
Galyna Dobrovolska, John Kim, Xiaoguang Ma, \textit{On the lower central series of an associative algebra (with an appendix by Pavel Etingof)}, Journal of Algebra \textbf{320} (2008), 213–237.  arXiv:0709.1905

\bibitem{Drenskybook}
V. Drensky, \textit{Free algebras and PI-algebras}. Graduate course in  algebra, Springer, Singapore, 1999.

\bibitem{Drensky84}
V. Drensky, \textit{Codimensions of $T$-ideals and Hilbert series of relatively free algebras}, Journal of Algebra \textbf{91}
(1984), 1--17.

\bibitem{EKM09}
Pavel Etingof, John Kim, Xiaoguang Ma, \textit{On universal Lie nilpotent associative algebras}, Journal of Algebra \textbf{321} (2009), 697--703. arXiv:0805.1909

\bibitem{FS07}
Boris Feigin, Boris Shoikhet, \textit{On [A,A]/[A, [A,A]] and on a $W_n$-action on the consecutive commutators of free associative algebras}, Mathematical Research Letters \textbf{14} (2007), 781--795. arXiv:math/0610410 

\bibitem{GZbook}
A. Giambruno, M. Zaicev, \textit{Polynomial identities and asymptotic methods}. Mathematical Surveys and Monographs, \textbf{122}. American Mathematical Society, Providence, RI, 2005.

\bibitem{Gordienko07}
A.S. Gordienko, \textit{Codimensions of commutators of length $4$}, Russian Mathematical Surveys \textbf{62} (2007), 187--188.

\bibitem{GTS10}
A.V. Grishin, L.M. Tsybulya, A.A. Shokola, \textit{On $T$-spaces and relations in relatively free Lie nilpotent associative algebras}, Journal of Mathematical Sciences (New York) \textbf{177} (2011), 868--877.

\bibitem{GL83}
Narain Gupta, Frank Levin, \textit{On the Lie ideals of a ring}, Journal of Algebra \textbf{81} (1983), 225--231.

\bibitem{Jennings47}
S.A. Jennings, \textit{On rings whose associated Lie rings are nilpotent}, Bulletin of the American Mathematical Society \textbf{53} (1947), 593--597.

\bibitem{JO13}
David Jordan, Hendrik Orem, \textit{An algebro-geometric construction of lower central series of associative algebras}, arXiv:1302.3992

\bibitem{Kerchev13}
George Kerchev, \textit{On the filtration of a free algebra by its associative lower central series}, Journal of Algebra \textbf{375} (2013), 322--327. arXiv:1101.5741

\bibitem{Kras97}
A.N. Krasil'nikov, \textit{On the semigroup nilpotency and the Lie nilpotency of associative algebras}, Mathematical Notes \textbf{62} (1997), 426--433.

\bibitem{Kras12}
Alexei Krasilnikov, \textit{The additive group of a Lie nilpotent associative ring}, Journal of Algebra \textbf{392} (2013), 10--22. arXiv:1204.2674

\bibitem{Latyshev65}
V.N. Latyshev, \textit{On finite generation of a $T$-ideal with the element $[x_1, x_2, x_3, x_4]$}, Siberian Mathematical Journal \textbf{6} (1965), 1432--1434. (in Russian)

\bibitem{Petrogradsky11}
V.M. Petrogradsky, \textit{Codimension growth of strong Lie nilpotent associative algebras}, Communications in Algebra \textbf{39} (2011), 918--928.

\bibitem{RT99}
D.M. Riley, V. Tasi\'c, \textit{Mal'cev nilpotent algebras}, Archiv der Mathematik (Basel) \textbf{72} (1999), 22--27.

\bibitem{Rowenbook}
L.H. Rowen, \textit{Polynomial Identities in Ring Theory}. Pure and Applied Mathematics, \textbf{84}, Acad. Press, New York-London, 1980.

\bibitem{Volichenko78}
I.B. Volichenko, \textit{The $T$-ideal generated by the element $[x_1, x_2, x_3, x_4]$}, Preprint no. 22, Institute of Mathematics of the Academy of Sciences of the Belorussian SSR, 1978. (in Russian)

\end{thebibliography}
\end{document}